\documentclass[a4paper]{article}

\usepackage[utf8]{inputenc}
\usepackage[width=15cm,height=22cm]{geometry}
\usepackage{amsmath,amsfonts,amssymb,amsthm,mathrsfs}

\usepackage[pdfusetitle]{hyperref}
\hypersetup{bookmarksdepth=subsection}
\usepackage[nameinlink,capitalise,noabbrev]{cleveref}
\hypersetup{
  colorlinks=true,
  linkcolor=blue!70!black,
  citecolor=red!70!black,
  urlcolor=blue!50!black,
}

\usepackage[dvipsnames]{xcolor}
\usepackage{bbm}
\usepackage{stmaryrd}
\usepackage{microtype}
\usepackage{enumitem}

\newtheorem{proposition}{Proposition}[section]
\newtheorem{theorem}[proposition]{Theorem}

\newtheorem{definition}[proposition]{Definition}

\newtheorem{corollary}[proposition]{Corollary}
\newtheorem{lemma}[proposition]{Lemma}
\newtheorem{remark}[proposition]{Remark}
\numberwithin{equation}{section}

\newcommand{\N}{\mathbb{N}}
\newcommand{\R}{\mathbb{R}}
\newcommand{\dd}{\,\mathrm{d}}
\newcommand{\B}{\mathcal{B}}
\newcommand{\Fac}{\operatorname{Fac}}
\newcommand{\Id}{\operatorname{Id}}
\newcommand{\norm}[1]{\lVert#1\rVert}
\newcommand{\bad}{\mathrm{bad}}
\newcommand{\good}{\mathrm{good}}

\newcommand{\irr}{\mathrm{irr}}

\newcommand{\Lie}{\operatorname{Lie}}
\newcommand{\id}{\operatorname{id}}
\newcommand{\eval}{\textnormal{\textsc{e}}}
\DeclareMathOperator{\Br}{Br}
\DeclareMathOperator{\ad}{ad}

\DeclareMathOperator{\vect}{span}
\newcommand{\supp}{\operatorname{supp}}
\newcommand{\A}{\mathcal{A}}
\newcommand{\seed}{\operatorname{seed}}
\newcommand{\shuff}{\mathbin{\sqcup\!\sqcup}}
\newcommand{\one}{\mathbf{1}}

\newcommand{\abs}[1]{\lvert #1 \rvert}
\newcommand{\pair}[2]{\langle #1,#2\rangle}
\newcommand{\intset}[1]{\llbracket #1 \rrbracket}

\newcommand{\Rg}{R_\good}
\newcommand{\Rb}{R_\bad}

\usepackage{etoolbox}
\newcounter{stepcount}
\newcommand{\step}[1]{\medskip

    \noindent\refstepcounter{stepcount}\emph{\textbf{Step~\arabic{stepcount}.} #1.}}
\AtBeginEnvironment{proof}{\setcounter{stepcount}{0}}

\title{Factor-parity Hall sets and controllability: \\ 
a classification of good and bad brackets}

\author{
Karine Beauchard\texorpdfstring{\thanks{Univ Rennes, CNRS, IRMAR - UMR 6625, F-35000 Rennes, France}}{},
Fr\'ed\'eric Marbach\texorpdfstring{\thanks{DMA, École normale supérieure, Université PSL, CNRS, 75005 Paris, France}}{}
}

\begin{document}

\maketitle

\begin{abstract}
    We introduce a class of Hall sets, which we call \emph{factor-parity} Hall sets, whose elements split into good and bad brackets.
    We prove that the good brackets can be steered simultaneously and arbitrarily in small time, which yields sufficient conditions for the small-time local controllability of control-affine systems.
    This positive result draws on constructions of Kawski, Agrachev--Gamkrelidze and Krastanov.
    Conversely, we prove that each bad bracket generates an obstruction to controllability, hence a family of necessary conditions.
\end{abstract}

\setcounter{tocdepth}{1}
\tableofcontents

\newpage

\section{Introduction}

\subsection{Small-time local controllability of control-affine systems}

In this article, we consider multi-input control-affine systems of the form
\begin{equation} 
    \label{eq:syst}
    \dot{x}(t) = f_0(x(t)) + u_1(t) f_1(x(t))  + \dotsb + u_q(t) f_q(x(t))
\end{equation}
where the state $x(t)$ belongs to $\R^d$ ($d \geq 1$), $u : [0,T] \to \R^q$ is the control, $f_0, f_1, \dotsc, f_q$ ($q \ge 1$) are real-analytic vector fields on a fixed open neighborhood $\Omega$ of $0$, and we assume that $f_0(0)=0$.

For each $T > 0$, each $u \in L^1((0,T);\R^q)$ and each $x^\circ \in \Omega$, the Cauchy problem \eqref{eq:syst} with initial condition $x(0)=x^\circ$ admits a unique maximal absolutely continuous solution.
We will consider small enough data so that this solution is defined up to time $T$.

In this article, we study the following standard notion of small-time local controllability of system \eqref{eq:syst} with controls small in $L^\infty$ (see e.g.~\cite[Definition 3.2]{Coron2007} or STLC$_\varepsilon$ in \cite{Kawski1987_Survey}).

\begin{definition}
    We say that \eqref{eq:syst} is \emph{$L^\infty$-STLC} when, for all $T,\rho > 0$, there exists $\delta > 0$ such that, for any $x^\circ, x^* \in \R^d$ with $\abs{x^\circ} + \abs{x^*} \leq \delta$, there exists $u \in L^\infty((0,T);\R^q)$ with $\norm{u}_{L^\infty} \le \rho$ such that the associated solution to \eqref{eq:syst} with initial condition $x(0) = x^\circ$ satisfies $x(T) = x^*$.
\end{definition}

\subsection{Algebraic notations and Lie brackets}

Let $X$ be an arbitrary set of unknowns.
We denote by $\Br(X)$ the free magma on $X$, by $\mathcal{L}(X)$ the free Lie algebra over $\R$ generated by $X$, and by $\eval$ the unique magma morphism $\Br(X) \to (\mathcal{L}(X), [\cdot,\cdot])$ extending $\id_X$.
In the sequel, implicitly, $\R$ is the base field of all vector spaces, algebras, Lie algebras, etc.
We refer to \cref{sec:def} or \cite[Section 2.1]{BeauchardLeBorgneMarbach2023} for more details on these standard objects.

In this paper, the notion of \emph{Hall set} plays a central role.
There are different conventions in the literature: one may decide to swap left and right factors, or swap the order, or both.
We follow Viennot's convention of \cite{Viennot1978} (also used in control theory by Sussmann \cite{Sussmann1986}).

\begin{definition}[Hall set]
    \label{def:Hall}
    A \emph{Hall set} on $X$ is a totally ordered subset $(\B,<)$ of $\Br(X)$ such that
    \begin{itemize}
        \item $X \subset \B$,
        \item for $a,b \in \Br(X)$, $(a,b) \in \B$ iff $a, b \in \B$, $a < b$ and either $b \in X$ or $b = (b', b'')$ with $b' \leq a$, 
        \item for every $a, b \in \B$ such that $(a,b) \in \B$, one has $a < (a,b)$.
    \end{itemize}
\end{definition}

The main interest of Hall sets is that their images under $\eval$ yield algebraic bases of $\mathcal{L}(X)$, called Hall bases, as proved in \cite[Corollary 1.1, Proposition 1.1 and Theorem 1.1]{Viennot1978}. 

\begin{proposition}
    \label{thm:viennot}
    Let $\B$ be a Hall set on $X$. 
    Then $\eval(\B)$ is a basis of $\mathcal{L}(X)$.
\end{proposition}

Our controllability conditions will be expressed using the iterated Lie brackets of the vector fields involved in \eqref{eq:syst}.
We use the following convention and notation.

\begin{definition}[Lie brackets of vector fields]
    We endow the vector subspace $C^\omega(\Omega;\R^d)$ of real-analytic vector fields on $\Omega$ with a Lie algebra structure by setting $[f,g] := (Dg) f - (Df) g$.
\end{definition}

\begin{definition}[Evaluation of Lie brackets]
    Let $X = \{X_0, X_1, \dotsc, X_q\}$.
    For $B \in \mathcal{L}(X)$, we denote by $f_B \in C^\omega(\Omega;\R^d)$ the image of $B$ under the unique Lie algebra homomorphism $\mathcal{L}(X) \to C^\omega(\Omega;\R^d)$ mapping $X_i$ to $f_i$. 
    For $b \in \Br(X)$, we write $f_b := f_{\eval(b)}$, so that $f_{X_i} = f_i$ and $f_{(a,b)} = [f_a, f_b]$ for $a, b \in \Br(X)$.
    Finally, $f_b(0) \in \R^d$ denotes its value at $x = 0$.
\end{definition}

\subsection{Statement of the main results}

In \cref{sec:def}, we define a class of Hall sets, which we call \emph{factor-parity} Hall sets.
Let us already mention that, starting from a partition $X = X_\good \sqcup X_\bad$ of the alphabet $X$, a factor-parity Hall set $\B$ on $X$ comes with a partition into \emph{good} and \emph{bad} brackets of the form $\B = \B_\good \sqcup \B_\bad$.

We consider a multi-input control-affine system of the form \eqref{eq:syst}.
We assume that it satisfies the Lie algebra rank condition:
\begin{equation}
    \label{eq:LARC}
    \Lie(f_0,f_1,\dotsc,f_q)(0) = \R^d.
\end{equation}

In the following statements, $\B$ is a factor-parity Hall set on $X = X_\good \sqcup X_\bad$ where $X_\good = \{ X_1, \dotsc, X_q \}$ and $X_\bad = \{ X_0 \}$.
Our main results are the following conditions.

\begin{theorem}[Sufficient condition]
    \label{thm:sufficient}
    If, for all $b \in \B_\bad$, $f_b(0) = 0$, then \eqref{eq:syst} is $L^\infty$-STLC.
\end{theorem}

\begin{theorem}[Necessary condition]
    \label{thm:necessary}
    Assume that there exists $b \in \B_\bad$ such that
    \begin{equation}
        \label{eq:syst-b-notin-supp-K}
        f_b(0) \notin \vect \{ f_a(0) \mid a \in \B_\good \}
    \end{equation}
    and
    \begin{equation}
        \label{eq:syst-no-other-bad}
        \forall a \in \B_\bad \setminus \{ b \}, \quad f_a(0) = 0.
    \end{equation}
    Then \eqref{eq:syst} is not $L^\infty$-STLC.
\end{theorem}

The complementarity between \cref{thm:sufficient,thm:necessary} settles an open question in control theory concerning the classification of brackets (see \cite[Section 4]{Kawski1987_Survey}).
Heuristically, they can be rephrased as follows:
\begin{itemize}
    \item if all bad brackets vanish, then the system is controllable;
    \item if all but one bad brackets vanish, and this single bad bracket is not compensated, then the system is not controllable.
\end{itemize}

\subsubsection{Relaxed sufficient condition}

The condition that all bad brackets vanish is unnecessarily strong.
Using rescaled controls of the form $\varepsilon^{1-\theta} u(\cdot/\varepsilon^\theta)$ as in \cite{Sussmann1987}, one easily obtains the following relaxed version where one allows bad brackets to be \emph{compensated} by good ones of lower \emph{weight}.
For a given $b \in \Br(X)$ and $\theta \in (0,1)$, set
\begin{equation}
    \label{eq:wtheta}
    \omega_\theta(b) := \theta n_0(b) + (n_1(b) + \dotsb + n_q(b))
\end{equation}
where $n_i(b)$ denotes the number of occurrences of the letter $X_i$ in $b$.

We will prove the following better sufficient condition.

\begin{theorem}
    \label{thm:relaxed-sufficient}
    Assume that there exists $\theta \in (0,1)$ such that, for all $b \in \B_\bad$,
    \begin{equation}
        \label{eq:fb0-compensated}
        f_b(0) \in \vect \{ f_a(0) \mid a \in \B_\good, \ \omega_\theta(a) < \omega_\theta(b) \}.
    \end{equation}
    Then \eqref{eq:syst} is $L^\infty$-STLC.
\end{theorem}

\subsubsection{Relaxed necessary condition}

The condition that all bad brackets except the considered one vanish is probably too stringent and might be somewhat relaxed.

However, there is an underlying difficulty which cannot be overlooked concerning the possibility to obtain $L^\infty$-STLC through the \emph{competition} of bad brackets.
This possibility had already been identified in \cite[Section 5]{Kawski1987_Survey}.
We introduced the following example, with a scalar-input control, in \cite[Proposition 1.21]{BeauchardMarbach2026_Quartic}:
\begin{equation} \label{eq:syst-magnitude}
    \begin{cases}
        \dot{x}_1 = u, \\
        \dot{x}_2 = x_1, \\
        \dot{x}_3 = x_2, \\
        \dot{x}_4 = x_1^2 x_3^2 - \lambda x_2^4.
    \end{cases}
\end{equation}
It is proved in \cite[Section 6]{BeauchardMarbach2026_Quartic} that there exists $\lambda^* > 0$ such that, for all $\lambda > \lambda^*$, this system is $L^\infty$-STLC.
Let $M_1 := (X_1,X_0)$, $M_2 := (M_1, X_0)$, $M_3 := (M_2, X_0)$, $Q_{1,3,3} := \ad_{M_2}^2 \ad_{X_1}^2(X_0)$ and $Q_{2,2,2} := \ad_{M_1}^4 (X_0)$.
Any factor-parity Hall set contains all these brackets, the first 3 being good, and the last two bad.
System \eqref{eq:syst-magnitude} is controllable thanks to a competition between these two bad quartic brackets.
Hence, a condition of the form \eqref{eq:syst-b-notin-supp-K} is insufficient to deny STLC.
It is mandatory to add some kind of assumption on the other bad brackets.
As a first step, we choose here the blunt formulation \eqref{eq:syst-no-other-bad}.

\subsection{Comparison with known results}

\subsubsection{Sufficient conditions}

We start with a famous sufficient condition due to Sussmann \cite{Sussmann1987}.
Recalling the notation $n_i(b)$ for the number of occurrences of the letter $X_i$ in $b$, we define a set of ``Sussmann-bad'' brackets as follows:
\begin{equation}
    S_\bad := \{ b \in \Br(X) \mid n_0(b) \text{ is odd and } n_i(b) \text{ is even for all } 1 \leq i \leq q \}.
\end{equation}
Sussmann proved the following result in  \cite{Sussmann1987}.

\begin{theorem}
    \label{thm:sussmann}
    Assume that $ \left\{ f_0, f_1, \dotsc, f_q \right\}$ satisfies the Lie algebra rank condition \eqref{eq:LARC}.
    Assume that there exists $\theta \in (0,1)$ such that, for the weight \eqref{eq:wtheta}, for all 
    $b \in S_\bad$,
    \begin{equation}
        f_b(0) \in \operatorname{span} \left\{ f_a(0) \mid a \in \Br(X),\ \omega_\theta(a) < \omega_\theta(b) \right\}.
    \end{equation}
    Then system \eqref{eq:syst} is $L^\infty$-STLC.
\end{theorem}

\begin{proof}
    See \cite[Section 7.3]{Sussmann1987} choosing $\Lambda_0$ to be the group of automorphisms generated only by the~$\sigma_i$ (not the $\pi$, $\tilde{\pi}$).
\end{proof}

Sussmann's sufficient condition exploits the input symmetries $u_i \mapsto -u_i$ (which dates back to the Hermes sufficient condition \cite{Hermes1982,Sussmann1983}) and $u \mapsto \check{u}$ (time reversal, which had been observed in particular cases by Stefani in \cite{Stefani1985}).
The brackets $b \in S_\bad$ are the ones whose coordinates are invariant under these natural substitutions.
Our sufficient condition \cref{thm:relaxed-sufficient} requires the compensation of far fewer brackets (see \cref{sec:compare-sussmann} for a detailed comparison).

In \cite{Kawski1987_Necessary}, Kawski observed that some brackets of $S_\bad$ could lead to controllable systems.
His proof introduced a new argument, which can be seen as the \emph{repetition} of a given control pattern concatenated multiple times with itself.
This idea started a fruitful line of work, including \cite{AgrachevGamkrelidze1993_Semigroups} by Agrachev and Gamkrelidze, \cite{Krastanov2009} by Krastanov, and the current work, which pushes the approach recursively, encoding it into a Hall set.

When expressed in a Hall set, our sufficient conditions \cref{thm:sufficient,thm:relaxed-sufficient} require the compensation of fewer bad brackets than Agrachev--Gamkrelidze or Krastanov (see \cref{sec:compare-AGK} for a detailed comparison).
Nevertheless, the core proof mechanism is similar, and our sufficient condition can be seen as an automatically-expanded version of their conditions.

\subsubsection{Necessary conditions}

Comparatively, less attention had been devoted to the derivation of necessary conditions for STLC, and one only knew obstructions related to very specific brackets or small families of brackets.
Some of these obstructions to controllability have also been observed for PDEs (see the survey \cite{Beauchard2026}).

\paragraph{Scalar-input case $q = 1$.}

Most known necessary conditions concern the scalar-input case $q = 1$.
Let us survey this case.
We fix some notation.
Let $M_0 := X_1$ and $M_{\nu+1} := (M_\nu, X_0)$ for $\nu \ge 0$.

Historically, the first obstruction to controllability discovered is the one associated with the bad bracket $W_1 := (X_1, (X_1, X_0)) = \ad_{X_1}^2(X_0) = \ad_{M_0}^2(X_0)$.
It is associated with the following necessary condition for $L^\infty$-STLC:
\begin{equation}
    f_{W_1}(0) \in \vect \{ f_a(0) \mid n_1(a) = 1 \}.
\end{equation}
See \cite[Proposition 6.3]{Sussmann1983} for the historical proof.

A first generalization of this condition was proved by Stefani in \cite{Stefani1985}.
She proved that, if \eqref{eq:syst} is $L^\infty$-STLC, then, for all $k \in \N^*$,
\begin{equation}
    f_{\ad^{2k}_{X_1}(X_0)}(0) \in \vect \{ f_a(0) \mid n_1(a) < 2k \}.
\end{equation}

In \cite{Kawski1987_Necessary}, Kawski proved that, if \eqref{eq:syst} is $L^\infty$-STLC, then, for $W_2 := \ad_{M_1}^2(X_0)$ and $P = \ad_{X_1}^3(X_0)$,
\begin{equation}
    f_{W_2}(0) \in \vect \{ f_a(0) \mid n_1(a) = 1 \text{ or } a = \ad_{X_0}^\nu \ad_{X_1}^3(X_0) \}.
\end{equation}

We generalized this condition by proving in \cite[Theorem 1.11]{BeauchardMarbach2026} the following necessary condition for $L^\infty$-STLC, conjectured by Kawski in \cite[p.\ 63]{Kawski1986}, for $W_k := \ad_{M_{k-1}}^2(X_0)$,
\begin{equation}
    f_{W_k}(0) \in \vect \{ f_a(0) \mid n_1(a) = 1 \text{ or } 2 < n_1(a) < 2k \}.
\end{equation}

In \cite[Section 3]{BeauchardMarbach2026}, we also constructed a Hall set $\B^\star$ for which we proved in \cite[Theorem 1.14]{BeauchardMarbach2026} the following necessary condition for $L^\infty$-STLC, with $P = \ad_{X_1}^3(X_0)$,
\begin{equation}
    f_{\ad^2_P(X_0)}(0) \in \vect \{ f_a(0) \mid a \in \B^\star,\ n_1(a) \leq 7 \text{ and } a \neq \ad_P^2(X_0) \}.
\end{equation}
For the same basis $\B^\star$, we proved in \cite[Section 7.1]{BeauchardMarbach2026_Quartic} many necessary conditions for STLC based on quartic brackets, i.e.\ brackets with $n_1 = 4$.
For example, we proved that, with $Q_{j,k,k} := \ad_{M_{k-1}}^2 \ad_{M_{j-1}}^2(X_0)$ for $1 \leq j \leq k$, and $k \leq 2j$, a necessary condition for $L^\infty$-STLC is that
\begin{equation}
    f_{Q_{j,k,k}}(0) \in \vect \{ f_a(0) \mid a \in \B^\star,\ n_1(a) \leq 2j+2k-1, \text{ and } a \neq Q_{j,k,k} \}.
\end{equation}

\emph{Note that all the brackets above belong to all Hall sets of the factor-parity class, and are all indeed classified as bad by these Hall sets.}

\paragraph{Multi-input case $q > 1$.}
Generally speaking, it is much harder to prove obstructions in the multi-input case $q > 1$ and the conditions are harder to state.
Let us mention two possible approaches.
In~\cite{GiraldiLissyMoreau2019}, the authors explore, when $q = 2$ and $f_2(0) = 0$, obstructions caused by the brackets $W_1$ and~$W_2$ above, by enlarging the compensating sets.
In~\cite{Gherdaoui2025_Obs_ODE}, the author explores obstructions linked with the simultaneous presence of the brackets $\ad_{X_i}^2(X_0)$ for $i = 1,\dotsc,q$.

\paragraph{Comparison.}
A strong advantage of our result \cref{thm:necessary} is that we have a classification of the full Hall set on $\{ X_0, X_1, \dotsc, X_q \}$, giving a general framework to identify bad brackets.
However, the non-compensation assumption \eqref{eq:syst-b-notin-supp-K} is very far from sharp.

\subsubsection{Classification results}

In \cite[Section 5]{BeauchardMarbach2026_Quartic}, we discussed the classification problem, dating back to \cite[Section 4]{Kawski1987_Survey}.
Using the Hall set $\B^\star$ introduced in \cite[Section 3]{BeauchardMarbach2026}, we proved that it could classify (in some precise sense) the brackets of $\Br(X)$ containing up to 4 occurrences of $X_1$.
The main strength of our current work is to produce such a classification for a full Hall set.

\subsection{Organization of the paper}

In \cref{sec:def}, we give definitions building up to the notion of \emph{factor-parity} Hall set.
In \cref{sec:good}, we prove the sufficient conditions of \cref{thm:sufficient,thm:relaxed-sufficient}.
In \cref{sec:observability,sec:support}, we prove the necessary condition of \cref{thm:necessary}.

\section{Definitions}
\label{sec:def}

Throughout this section, $X$ is an arbitrary set of unknowns.

\subsection{Free magma and factorization} 
\label{s:hall-sets}

Recall that we denote by $\Br(X)$ the free magma on $X$.
We will use the following associated notions.

For $b \in \Br(X)$, $|b|$ denotes the length of $b$.
If $|b| > 1$, $b$ can be written in a unique way as $b = (b', b'')$, with $b', b'' \in \Br(X)$. 
We use the notations $\lambda(b) = b'$ and $\mu(b) = b''$, which define maps $\lambda,\mu: \Br(X)\setminus X \to \Br(X)$.
We also denote by $\deg_X(b) \in \N^{(X)}$ the multidegree of $b$, $\supp \deg_X(b) \subset X$ the set of letters occurring in~$b$ and $\deg_x(b) \in \N$ the number of occurrences of the letter $x$ in $b$.
For $a \in \Br(X)$, let $\ad_a : \Br(X) \to \Br(X)$ be the map defined by $\ad_a(b) := (a,b)$.

For example, when $X = \{ X_0, X_1, X_2 \}$, $a = \ad_{X_1}^2(X_0) = (X_1,(X_1,X_0)) \in \Br(X)$ satisfies $|a| = 3$, $\lambda(a) = X_1$, $\mu(a) = (X_1, X_0)$ and $\deg_X(a) = \{ X_0 \mapsto 1, X_1 \mapsto 2, X_2 \mapsto 0 \}$ so that $\supp \deg_X(a) = \{ X_0, X_1 \}$ and $\deg_{X_1}(a) = 2$.

\begin{definition}[Factors]
    \label{def:factors}
    For $b \in \Br(X)$, we define by induction on length its \emph{factors} as
    \begin{equation}
        \Fac(b) :=
        \begin{cases}
            \varnothing, & \text{if } b \in X, \\
            \{ a \} \cup \Fac(c), & \text{if } b = (a,c).
        \end{cases}
    \end{equation}
    For $B \subset \Br(X)$, we write $\Fac(B) := \bigcup_{b \in B} \Fac(b)$ and we call $B$ \emph{factor-stable} when $\Fac(B) \subset B$.
\end{definition}

\begin{definition}[Factor closure]
    Given $B \subset \Br(X)$, let $\Fac^+(B)$ denote the smallest factor-stable subset of $\Br(X)$ containing $\Fac(B)$.
    We also let $\Fac^*(B) := B \cup \Fac^+(B)$.
    For a singleton, $\Fac^+(b) := \Fac^+(\{b\})$ and $\Fac^*(b) := \Fac^*(\{b\})$.
\end{definition}

\begin{lemma}
    \label{lem:Fac+ES}
    If $S, E \subset \Br(X)$ are such that $S$ is factor-stable and $\Fac(E) \subset S$, then $\Fac^+(E) \subset S$.
\end{lemma}

\begin{lemma}[Factorization and seed]
    \label{lem:factorization}
    For $b \in \Br(X)$, there exists a unique $r \in \N$, $m_1, \dotsc, m_r \geq 1$, $a_1, \dotsc, a_r \in \Br(X)$ with $a_j \neq a_{j+1}$ for $1 \leq j < r$, and $X_i \in X$ such that
    \begin{equation}
        \label{eq:b-factorization}
        b = \ad_{a_r}^{m_r} \dotsb \ad_{a_1}^{m_1} (X_i).
    \end{equation}
    We call this expression of $b$ its \emph{factorization}, and one has 
    \begin{equation}
        \label{eq:nested}
        \Fac(b) = \{ a_1, \dotsc, a_r \}
        \quad \text{and} \quad
        \Fac^*(b) = \{ b \} \cup \bigcup_{1 \leq j \leq r} \Fac^*(a_j).
    \end{equation}
    We call $X_i$ the \emph{seed} of $b$ and we use the notation $\seed(b) := X_i \in X$.
    
    When $b \in X$, its factorization is just $b = X_i$, so $r = 0$, $\Fac(b) = \varnothing$ and $\seed(b) = b$.
\end{lemma}

\subsection{Factorization in Hall sets}

Let $\B$ be a Hall set on $X$.

\begin{lemma}[Factorization in a Hall set]
    \label{lem:facto-Hall}
    For all $b \in \B$, the factors of $b$ given in \eqref{eq:b-factorization} belong to~$\B$ and satisfy $a_1 < \dotsb < a_r < b$ and $a_1 < \seed(b)$.
\end{lemma}

\begin{proof}
    The fact that $a_i \in \B$ and the inequality chain $a_1 < \dotsb < a_r$ follow from the second Hall axiom.
    By the third Hall axiom, $a_r < b$.
    Since the innermost bracket $(a_1, \seed(b)) \in \B$, $a_1 < \seed(b)$ by the second Hall axiom.
\end{proof}

\begin{lemma}
    \label{lem:B<p-stable}
    Let $p \in \B$.
    Then $\B_{<p} := \{ b \in \B \mid b < p \}$ is factor-stable.
\end{lemma}

\begin{proof}
    Let $b \in \B_{<p}$ and $a \in \Fac(b)$.
    By \cref{lem:facto-Hall}, $a < b$. 
    Hence $a < p$ and $a \in \B_{<p}$.
\end{proof}

\begin{lemma}[Minimum of the factor closure]
    \label{lem:min-fac}
    Let $b \in \B$.
    Then
    \begin{enumerate}[label=\textup{(\roman*)},leftmargin=2em]
        \item \label{it:min-fac-1}
        \(\supp\deg_X(a)\subset\supp\deg_X(b)\) for every \(a\in\Fac^*(b)\);
        \item \label{it:min-fac-2}
        \(\min\Fac^*(b)\in X\);
        \item \label{it:min-fac-3}
        \(\min\Fac^*(b)=\min\supp\deg_X(b)\).
    \end{enumerate}
\end{lemma}

\begin{proof}
    (i) By additivity of the multidegree, $\deg_x(a) \leq \deg_x(b)$ for any $x \in X$.

    (ii) By contradiction, assume that $x := \min \Fac^*(b)$ is not a letter.
    Then $\Fac(x) \neq \varnothing$. 
    Let $a \in \Fac(x)$.
    Since $\Fac^*(b)$ is factor-stable, $a \in \Fac^*(b)$.
    By \cref{lem:facto-Hall}, $a < x$, a contradiction.
    
    (iii) Proceed by induction on \(\abs{b}\), the case \(b\in X\) being trivial.  Let \(b\) be as in \eqref{eq:b-factorization}.
    Since \(\min\Fac^*(a_j)\leq a_j<b\) by \cref{lem:facto-Hall}, \eqref{eq:nested} gives
    \begin{equation}
        \min\Fac^*(b)=\min_j\min\Fac^*(a_j)=\min_j\min\supp\deg_X(a_j),
    \end{equation}
    the second equality by the induction assumption.
    
    By \cref{lem:facto-Hall}, $a_1 < \seed(b)$, whence \(\min\Fac^*(b)\leq a_1<\seed(b)\).
    Therefore
    \begin{equation}
        \min\supp\deg_X(b)
        =\min\Bigl(\{\seed(b)\}\cup\bigcup_j\supp\deg_X(a_j)\Bigr)
        =\min_j\min\supp\deg_X(a_j),
    \end{equation}
    which concludes the proof.
\end{proof}

\subsection{Expansion in a basis given by a Hall set} 

\begin{definition}[Support]
    Let $\B$ be a Hall set on $X$.
    For $z \in \mathcal{L}(X)$ and $b \in \B$, we write $\pair{z}{b}_\B$ for the coordinate of $z$ along $b$ in the basis $\eval(\B)$. 
    We define
    \begin{equation}
        \supp_\B (z) := \left\{ b \in \B \mid \langle z, b \rangle_{\B} \neq 0 \right\}.
    \end{equation}
    For a subset $V \subset \mathcal{L}(X)$, we write $\supp_\B(V) := \bigcup_{z \in V} \supp_\B(z)$.
\end{definition}

The following structural property (proved in \cite[Theorem 2.1]{BeauchardLeBorgneMarbach2022}) follows from the classical recursive rewriting algorithm on Hall sets, described for example in \cite[Section 2.1]{BeauchardLeBorgneMarbach2022} or \cite[Section 9]{Reutenauer2003}.

\begin{lemma} 
    \label{lem:left_fact}
    Let $\B$ be a Hall set on $X$ and $a < b \in \B$. 
    Then, either $(a,b)\in\B$, or all elements $c \in \supp_{\B} [a,b]$ satisfy $\lambda(c) > a$.
    In both cases, for all $c \in \supp_\B [a,b]$, $\lambda(c) \geq a$.
\end{lemma}

\subsection{Hall order}

As in \cite[Section 1.4.2]{BeauchardLeBorgneMarbach2022}, one can construct Hall sets from given orders on subsets of $\Br(X)$.

\begin{definition}
    We say that $H \subset \Br(X)$ is $\lambda$-stable when, for all $b \in H \setminus X$, $\lambda(b) \in H$.
\end{definition}

\begin{definition}
    \label{def:Hall-order}
    Let $H$ be a $\lambda$-stable subset of $\Br(X)$.
    We say that $<$ is a \emph{Hall order} on $H$ when it is a total order such that, for all $b \in H \setminus X$, one has $\lambda(b) < b$.
\end{definition}

\begin{lemma}
    \label{lem:Hall-exists-from-order}
    Let $H$ be a $\lambda$-stable subset of $\Br(X)$ with $X \subset H$, endowed with a Hall order $<$ such that, for all $a < b \in H$, one has $(a, b) \in H$.
    There exists a unique Hall set $\B \subset H$ over~$X$ associated with this order.
\end{lemma}

\begin{proof}
    We show by induction that, for each $\ell \geq 1$, the set $\B_\ell$ of elements of length $\ell$ of $\B$ is uniquely determined by the Hall axioms and the given order.
    By the first item of \cref{def:Hall}, $\B_1 = X$.
    Then, for $\ell \geq 2$, by the second item of \cref{def:Hall},
    \begin{equation}
        \B_\ell = \{ (a,b) \mid a \in \B_{\ell-1}, b \in \B_1, a < b \}
        \cup \bigcup_{1 \leq j \leq \ell - 2} \{ (a,b) \mid a \in \B_j, b \in \B_{\ell-j}, \lambda(b) \leq a < b \}. 
    \end{equation}
    The assumption that $<$ is a Hall order yields the third item of \cref{def:Hall}.
\end{proof}

\subsection{Dichotomic Hall set}

For our applications, it is helpful to construct Hall sets in which all ``bad'' brackets are ordered after all ``good'' brackets.
Given a set $A \subset \Br(X)$ of brackets identified as potentially bad, we prove the existence of such a Hall set.
We define $(A,A) := \{ (a, b) \mid  a, b \in A \} \subset \Br(X)$.

\begin{proposition} 
    \label{prop:Hall-dichotomic}
    Let $A \subset \Br(X)$ such that $(A, A) \subset A$.
    Let $\prec$ be a Hall order on $\Br(X)$.
    There exists a unique Hall set $(\B,<)$ such that, for all $a, b \in \B$, $a < b$ if and only if
    \begin{enumerate}[label=(\roman*)]
        \item \label{it:dicho-1}
        either $a \notin A$ and $b \in A$,
        
        \item \label{it:dicho-2}
        or both $a, b \in A$ (or both $a, b \notin A$) and $a \prec b$.
    \end{enumerate}
\end{proposition}

\begin{proof}
    Let $H$ be the subset of $\Br(X)$ whose elements are the $b \in \Br(X)$ such that, for all $k \in \N$ such that $\lambda^k(b) \notin X$ and $\lambda^{k+1}(b) \in A$, one has $\mu \lambda^k(b) \in A$.
    Then $H$ is a $\lambda$-stable subset of $\Br(X)$ with $X \subset H$.
    We endow $H$ with a total order by setting $a < b$ if and only if \ref{it:dicho-1} or \ref{it:dicho-2} holds.

    \medskip \noindent 
    \emph{Step 1: We prove that $<$ is a Hall order on $H$.} 
    Let $b \in H \setminus X$. 
    Since $\prec$ is a Hall order, $\lambda(b) \prec b$.
    If $\lambda(b) \notin A$ and $b \in A$, then $\lambda(b)<b$ by \ref{it:dicho-1}.
    If $\lambda(b) \notin A$ and $b \notin A$, then $\lambda(b)<b$ by \ref{it:dicho-2}.
    If $\lambda(b) \in A$ then $\mu(b) \in A$ because $b \in H$, thus $b \in (A,A) \subset A$ and $\lambda(b) < b$ by \ref{it:dicho-2}.
    
    \medskip \noindent 
    \emph{Step 2: We prove that, for all $a < c \in H$, $(a, c)\in H$.}
    Let $a < c \in H$ and $b = (a, c)$.
    We want to prove that for all $k \in \N$ such that $\lambda^k(b) \notin X$ and $\lambda^{k+1}(b) \in A$, one has $\mu \lambda^k(b) \in A$. 
    This holds for $k \geq 1$ because $a \in H$. 
    It also holds for $k=0$: indeed, if $a \in A$ then $c \in A$ by \ref{it:dicho-1} because $a < c$.

    \medskip \noindent 
    \emph{Step 3: Conclusion.} The existence and uniqueness follow by \cref{lem:Hall-exists-from-order}.

    Indeed, any Hall set satisfying \ref{it:dicho-1} and \ref{it:dicho-2} must be included in $H$.
    Let $\B$ be such a Hall set, $b \in \B$ and $k \in \N$ such that $\lambda^k(b) \notin X$ and $\lambda^{k+1}(b) \in A$.
    Since $\lambda^k(b) \in \B$, by the second Hall axiom, $\lambda^{k+1}(b) < \mu \lambda^k(b)$.
    By \ref{it:dicho-1}, $\mu \lambda^k(b) \in A$.
    So $b \in H$.
\end{proof}

\begin{proposition}
    \label{prop:Hall-dichotomic-bis}
    Let $A \subset \Br(X)$ such that $(A, A) \subset A$.
    Let $\prec$ be a Hall order on $\Br(X)$.
    Let~$\B$ be the Hall set given by \cref{prop:Hall-dichotomic}.
    Define $\B_\bad := \B \cap A$ and $\B_\good := \B \setminus A$.
    Then 
    \begin{itemize}
        \item $\B_\good < \B_\bad$, i.e.\ $a < b$ for all $a \in \B_\good$ and $b \in \B_\bad$,
        \item $\B_\good$ is factor-stable,
        \item if $a, b \in \B_\bad$, then $\supp_\B [a,b] \subset \B_\bad$.
    \end{itemize}
\end{proposition}

\begin{proof}
    Let us prove each claim.
    \begin{itemize}
        \item The fact that $\B_\good < \B_\bad$ is a mere rewriting of item \ref{it:dicho-1}.
        
        \item Let $b \in \B_\good$. 
        Write its factorization \eqref{eq:b-factorization}.
        By \cref{lem:facto-Hall}, $a_1 < \dotsb < a_r < b$.
        Since $b \in \B_\good$, by the first item, $a_1, \dotsc, a_r \in \B_\good$.
        
        \item Let $a, b \in \B_\bad$ and $c \in \supp_\B [a,b]$.
        By \cref{lem:left_fact}, $\lambda(c) \geq \min (a,b)$.
        By the third Hall set axiom of \cref{def:Hall}, $c > \lambda(c)$.
        By the first item, $c \in \B_\bad$.
        \qedhere
    \end{itemize}
\end{proof}

\subsection{Factor-parity Hall sets}
\label{sec:of}

Motivated by the analysis of the positivity properties of coordinates of the second kind (see \cref{Subsec:Coord2_Bstar}), we introduce the following family of dichotomic Hall sets.
We start from a partition
\begin{equation} 
    \label{eq:partition}
    X = X_\good \sqcup X_\bad.
\end{equation}

\begin{definition}[Factor-parity Hall set]\label{def:of}
    We say that a Hall set \(\B\) over \(X\) is of the \emph{factor-parity} class for \eqref{eq:partition} when it splits as \(\B=\B_\good\sqcup\B_\bad\) with
    \begin{itemize}
        \item \(\B_\good<\B_\bad\);
        \item for every \(b\in\B\), with the notation of \eqref{eq:b-factorization},
        \begin{equation}\label{eq:of}
            b\in\B_\good
            \iff
            \Fac(b)\subset\B_\good
            \text{ and }
            \bigl(\seed(b)\in X_\good \text{ or } m_j \text{ odd for some } j\bigr).
        \end{equation}
    \end{itemize}
\end{definition}

\begin{remark}
    \label{rem:of}
    Applying \eqref{eq:of} to letters gives $X \cap \B_\good = X_\good$ and $X \cap \B_\bad = X_\bad$.
\end{remark}

\begin{proposition}[Existence]
    \label{prop:of-existence}
    Let \eqref{eq:partition} be a partition of $X$ and $\prec$ be a Hall order on $\Br(X)$.
    There exists a unique factor-parity Hall set $(\B,<)$ on $X$ such that for all $a,b\in\B_\good$ (or both in $\B_\bad$), $a<b$ if and only if $a \prec b$.
\end{proposition}

\begin{proof}
    We define a subset $H\subset\Br(X)$ by induction on length as follows.
    For $b\in\Br(X)$, writing its factorization as in \eqref{eq:b-factorization}, declare that
    \begin{equation}
        \label{eq:intrinsic-good}
        b\in H
        \quad\Longleftrightarrow\quad
        a_1,\dotsc,a_r\in H
        \quad\text{and}\quad
        (\seed(b) \in X_\good) \lor
        (\exists j
        \text{ such that }m_j\text{ is odd}).
    \end{equation}
    Set $A := \Br(X) \setminus H$.
    If $b = (a, c) \in (A, A)$ then $a \in \Fac(b)$ and $a \notin H$, so $b \notin H$ by \eqref{eq:intrinsic-good}.
    Hence $(A, A) \subset A$.
    Let $(\B,<)$ be the Hall set given by \cref{prop:Hall-dichotomic,prop:Hall-dichotomic-bis}.
    Let $\B_\good := \B \setminus A$ and $\B_\bad := \B \cap A$.
    Since $\B_\good = \B \cap H$, \eqref{eq:intrinsic-good} yields \eqref{eq:of}.
\end{proof}

\begin{corollary}
    There exists a factor-parity Hall set for any partition \eqref{eq:partition}.
\end{corollary}

\begin{proof}
    It suffices to choose an arbitrary Hall order on $\Br(X)$ and apply \cref{prop:of-existence}.
\end{proof}

\subsection{Coordinates of the second kind} \label{Subsec:Coord2_Bstar}

Here $X = \{ X_0, X_1, \dotsc, X_q \}$ and for $b \in \Br(X)$, $n_i(b)$ is the number of occurrences of $X_i$ in $b$.

\begin{definition}[Drift control]
    To lighten some formulas, we use the notation $u_0 \equiv 1$ to denote a virtual constant control.
\end{definition}

\begin{definition}\label{Def:Coord2}
    Let $\B$ be a Hall set on $X$.
    The coordinates of the second kind associated to~$\B$ are the unique family $(\xi_b)_{b\in\B}$ of functionals $\R_+ \times L^1_{\mathrm{loc}}(\R_+;\R^q) \rightarrow \R$ defined by induction in the following way: for all $t \ge 0$ and $u \in L^1_{\mathrm{loc}}(\R_+;\R^q)$,
    \begin{itemize}
        \item $\xi_{X_i}(t,u) = \int_0^t u_i(\tau) \dd \tau$ for $0 \leq i \leq q$,  
        \item for $b \in \B \setminus X$, there is a unique pair $(a,c)$ of elements of $\B$ such that $a < c$ and a unique maximal integer $m \in\N^*$ such that $b = \ad_a^m (c)$ and then
        \begin{equation} \label{eq:def:coord2}
            \xi_b(t,u) := \frac{1}{m!} \int_0^t \xi_a^m(\tau,u) \dot{\xi}_c(\tau,u) \dd\tau.
        \end{equation}
    \end{itemize}
    The coordinates $(\xi_b)_{b \in \B}$ are absolutely continuous. 
    In the sequel, equalities involving $(\dot{\xi}_b)_{b \in \B}$ are implicitly understood to hold almost everywhere.
\end{definition}

\begin{lemma}
    \label{lem:coord2-rec}
    Let $\B$ be a Hall set on $X$.
    For all $b \in \B$,
    \begin{equation}
        \xi_b(t,u) = \int_0^t \frac{(\xi_{a_r}(s,u))^{m_r}}{m_r!} \dotsb \frac{(\xi_{a_1}(s,u))^{m_1}}{m_1!} u_i(s) \dd s
    \end{equation}
    where $b$ is given by the factorization \eqref{eq:b-factorization} and $X_i = \seed(b)$.
\end{lemma}

\begin{lemma}[Homogeneity]
    \label{lem:xi-homogeneity}
    Let $\B$ be a Hall set on $X$.
    Let $b\in\B$, $T>0$ and $u\in L^1((0,T);\R^q)$.
    For $\lambda \in \R^q$, set $\lambda u:=(\lambda_1u_1,\dotsc,\lambda_qu_q)$. 
    Then, for all $0 \leq t \leq T$,
    \begin{equation}
        \xi_b(t,\lambda u)
        =
        \left(\prod_{i=1}^q\lambda_i^{n_i(b)}\right)\xi_b(t,u).
        \label{eq:xi-control-homogeneity}
    \end{equation}
    Moreover, for $T' > 0$, letting $u^{T,T'} := \frac{T}{T'} u(\frac{T}{T'}\cdot)$, one has, for $0 \leq t' \leq T'$,
    \begin{equation}
        \xi_b(t', u^{T,T'})
        =
        \left(\frac{T'}{T}\right)^{n_0(b)}
        \xi_b\left(\frac{T}{T'}t',u\right).
        \label{eq:xi-time-homogeneity}
    \end{equation}
\end{lemma}

\begin{proof}
    Both identities follow by induction on length. 
    They are immediate for the letters. 
    If $b=\ad_a^m(c)$, the defining identity \eqref{eq:def:coord2} and the equalities $n_i(b)=mn_i(a)+n_i(c)$ for $0 \le i \le q$ give the induction step.
\end{proof}

\begin{lemma}[Concatenation]
    \label{lem:concatenation-factor-null}
    Let $\B$ be a Hall set on $X$.
    Let $S$ be a factor-stable subset of $\B$.
    Let $T > 0$ and $u \in L^1((0,T);\R^q)$ such that, for all $a \in S$, $\xi_a(T,u) = 0$.
    
    Let $b \in \B$ such that $\Fac(b) \subset S$.
    For all $T' > 0$, $v \in L^1((0,T');\R^q)$ and $t' \in [0,T']$,
    \begin{equation}
        \label{eq:xi-concatenation}
        \xi_b(T + t', u \diamond v) = \xi_b(T, u) + \xi_b(t', v)
        \quad \text{and} \quad
        \dot{\xi}_b(T + t', u \diamond v) = \dot{\xi}_b(t', v).
    \end{equation}
\end{lemma}

\begin{proof}
    Start by proving that $\xi_a(T+t',u\diamond v) = \xi_a(t',v)$ for all $a \in S$ and $t' \in [0,T']$.
    One can proceed by induction on $|a|$, using \cref{Def:Coord2} for the base case $a \in X$ and \cref{lem:coord2-rec} for the induction.
    Then \eqref{eq:xi-concatenation} follows from \cref{lem:coord2-rec}.
\end{proof}

\begin{definition}[Canonical system]
    \label{def:canonical}
    Let $\B$ be a Hall set on $X$ and $H \subset \B \setminus \{ X_0 \}$ a finite factor-stable subset.
    We denote by~$\Sigma_H$ the polynomial system on $\R^H$ defined by
    \begin{equation}
        \dot{x}_b
        =
        \frac{x_{a_r}^{m_r}}{m_r!}\dotsb
        \frac{x_{a_1}^{m_1}}{m_1!}u_i
        \qquad
        \text{ for } \quad
        b=\ad_{a_r}^{m_r}\dotsb\ad_{a_1}^{m_1}(X_i)
        \quad 
        \text{ as in \eqref{eq:b-factorization}}.
        \label{eq:Hall-H}
    \end{equation}
    Its solution starting from the origin is the family $(\xi_b(t,u))_{b\in H}$.
    
    When $\Sigma_H$ is written as a control-affine system of the form \eqref{eq:syst}, one has the canonical identity (see \cite[Proposition~4.25]{BeauchardLeBorgneMarbach2026}):
    \begin{equation}
        \label{eq:Hall-H-Ev}
        \forall b \in \B, \quad 
        f_b(0) = 
        \begin{cases}
            \partial_{x_b} & \text{if } b \in H, \\
            0 & \text{otherwise.}
        \end{cases}
    \end{equation}
\end{definition}

\newpage 

\section{Proof of the sufficient condition}
\label{sec:good}

Let $\B$ be a factor-parity Hall set on $X = \{ X_0, X_1, \dotsc, X_q \}$ with $X_\good = \{ X_1, \dotsc, X_q \}$ and $X_\bad = \{ X_0 \}$.
We prove \cref{thm:relaxed-sufficient} (which of course implies \cref{thm:sufficient}).

By classical arguments, it suffices to prove the following result.

\begin{theorem}
    \label{thm:bss-good}
    Let $G \subset \B_\good$ be finite and factor-stable.
    Then the canonical system $\Sigma_G$ of \cref{def:canonical} is small-time globally controllable.
    More precisely, for all $T > 0$ and $x^\circ, x^* \in \R^G$, there exists $u \in L^\infty((0,T);\R^q)$ such that the associated solution to \eqref{eq:Hall-H} satisfies $x(T;u,x^\circ) = x^*$.
\end{theorem}

\begin{proof}[Proof of \cref{thm:relaxed-sufficient}]
    We assumed that \eqref{eq:syst} satisfies the Lie algebra rank condition \eqref{eq:LARC}.
    Since bad brackets are compensated by assumption \eqref{eq:fb0-compensated}, one has $\vect \{ f_b(0) \mid b \in \B_\good \} = \R^d$.
    Hence there exists a finite subset $G \subset \B_\good$ such that $\vect \{ f_b(0) \mid b \in G \} = \R^d$ and, for all $b \in \B \setminus G$, the compensation \eqref{eq:fb0-compensated} holds.
    By \eqref{eq:of}, $\B_\good$ is factor-stable, so, up to replacing $G$ with $\Fac^*(G)$, we can assume that $G$ is both finite and factor-stable.
    By \cref{thm:bss-good}, $\Sigma_G$ is small-time globally controllable.
    In particular, it admits a \emph{dual family} in the sense of \cite[Section A.2]{BeauchardMarbach2026_Quartic}.
    Thus, using the classical Sussmann control rescaling as in \cite[Theorem A.14]{BeauchardMarbach2026_Quartic}, we conclude that \eqref{eq:syst} is $L^\infty$-STLC. 
\end{proof}

In this section, we therefore focus on proving \cref{thm:bss-good}.
We recall the following classical result from control theory concerning normal accessibility.

\begin{proposition}
    \label{prop:normal-accessibility}
    Assume that $\{ f_0, f_1, \dotsc, f_q \}$ satisfies the Lie algebra rank condition \eqref{eq:LARC}.
    
    For all $T,\rho,\delta > 0$, there exists $v \in L^\infty((0,T);\R^q)$ such that $\norm{v}_{L^\infty} < \rho$, $\abs{x(t;v)} < \delta$ for all $t \in [0,T]$, and such that the end-point map 
    \begin{equation}
        \mathcal{E}_T :
        \begin{cases}
            L^\infty((0,T);\R^q) & \to \R^d, \\
            u & \mapsto x(T;u)
        \end{cases}
    \end{equation}
    has a surjective differential at $v$, where $x(t;u)$ denotes the solution to \eqref{eq:syst} with initial condition $x(0) = 0$ and control $u$.
\end{proposition}

\subsection{Description of the induction mechanism}

To prove \cref{thm:bss-good}, we proceed by induction on the number of distinct factors $|\Fac(G)| \ge 0$ involved in $G$.
The assumption that $G$ is factor-stable implies that
\begin{equation}
    \Fac^+(G) = \Fac(G) \subset G.
\end{equation}
The initialization is immediate.

\begin{lemma}
    \label{lem:bof-init}
    For $G \subset \B_\good$ with $|\Fac(G)| = 0$, $\Sigma_G$ is small-time globally controllable.
\end{lemma}

\begin{proof}
    Since $\Fac(G) = \varnothing$ and $X_0 \in \B_\bad$, one has $G \subset X_\good$.
    Let $T > 0$, $x^\circ, x^* \in \R^G$.
    For each $X_i \in G$, use the explicit constant control $u_i := (x^*_i - x^\circ_i) / T$, and $u_i = 0$ when $X_i \notin G$.
\end{proof}

We now assume that $\Fac(G) \neq \varnothing$.
To perform the induction, we single out the largest factor of elements of $G$, according to the underlying Hall order.
We set
\begin{equation}
    \label{eq:bof-p}
    p := \max \Fac(G) \in G.
\end{equation}
We partition $G$ as follows:
\begin{equation}
    \label{eq:G-partition}
    G = L \sqcup Z,
    \quad \text{where} \quad
    Z := \{ b \in G \mid p \in \Fac(b) \}
    \quad \text{and} \quad 
    L := G \setminus Z.
\end{equation}
The Hall order axioms entail the following elementary consequences.

\begin{lemma}
    \label{lem:bof-G-part}
    One has:
    \begin{enumerate}[label=(\roman*)]
        \item \label{it:bof-G-part-1} for all $b \in Z$, $p < b$;
        \item \label{it:bof-G-part-2} $p \in L$;
        \item \label{it:bof-G-part-3} $\Fac(L) \subset L \cap \B_{<p}$.
    \end{enumerate}
\end{lemma}

\begin{proof}
    First, for $b \in Z$, since $p \in \Fac(b)$, \cref{lem:facto-Hall} yields $p < b$.
    Second, since $p \in G = Z \sqcup L$, $p \in L$ because $p \not < p$. 
    Third, let $\ell \in L$ and $a \in \Fac(\ell) \subset \Fac(G) \subset G$.
    By \eqref{eq:bof-p}, $a \leq p$.
    Moreover $a \neq p$ (otherwise $\ell \in Z$ by \eqref{eq:G-partition}).
    Thus $a < p$ and $a \in \B_{<p}$.
    By \ref{it:bof-G-part-1}, $a \notin Z$.
    So $a \in L$.
\end{proof}

\begin{corollary}
    \label{lem:L}
    $L \subset \B_\good$ is factor-stable, $\Fac(L) \subset \Fac(G) \setminus \{ p \}$ and $|\Fac(L)| < |\Fac(G)|$.
\end{corollary}

The heart of the induction mechanism is the following result proved in \cref{subsec:vertical-loops}.

\begin{proposition}
    \label{prop:vertical-loops}
    Let $\kappa \geq 1$. 
    Assume that \cref{thm:bss-good} holds for all sets satisfying $|\Fac(G)| < \kappa$. 
    Let $G\subset\B_\good$ be finite, factor-stable, and such that $|\Fac(G)| = \kappa$.
    Then, for all $T>0$,
    \begin{equation}
        \label{eq:vertical-loops}
        \Big\{
            (\xi_b(T,u))_{b \in Z}
            \mid
            u \in L^\infty((0,T);\R^q) \text{ s.t., for all }
            a \in G \setminus Z, \enskip
            \xi_a(T,u)=0
        \Big\}
        = \R^{Z}.
    \end{equation}
\end{proposition}

\cref{prop:vertical-loops} expresses that, for any $T > 0$, starting from the origin, one can find \emph{control loops} reaching any desired target along the coordinates of $Z$ at time $T$, while driving the coordinates of $L$ back to~$0$.
Once such loops are available, the induction follows easily.

\begin{proof}[Proof of \cref{thm:bss-good}]
    We proceed by induction on $|\Fac(G)| \geq 0$.
    The case $|\Fac(G)| = 0$ is covered in \cref{lem:bof-init}.
    Let $\kappa \geq 1$, and assume that the result holds when $|\Fac(G)| < \kappa$. 
        
    Let $G \subset \B_\good$ with $|\Fac(G)| = \kappa$.
    We decompose the state $x \in \R^G = \R^L \times \R^Z$ as $x = (x_L, x_Z)$.
    Let $T > 0$, $x^\circ, x^* \in \R^G$.
    By \cref{lem:L}, we can apply the induction hypothesis to $L$.
    Thus there exist controls $u^\circ, u^* \in L^\infty((0,T);\R^q)$ such that
    \begin{equation}
        x_L(T;u^\circ,x^\circ_L) = 0_L
        \quad \text{and} \quad 
        x_L(T;u^*,0_L) = x^*_L.
    \end{equation}
    Given a control $u^{loop} \in L^\infty((0,T);\R^q)$ such that $x_L(T;u^{loop},0_L) = 0_L$, define the concatenation $u := u^\circ \diamond u^{loop} \diamond u^* \in L^\infty((0,3T);\R^q)$.
    The triangular nature of the system \eqref{eq:Hall-H} and the concatenation property \cref{lem:concatenation-factor-null} imply that
    \begin{equation}
        x_L(3T;u,x^\circ_L) = x^*_L
        \quad \text{and} \quad 
        x_Z(3T;u,x^\circ) = x_Z(T;u^\circ,x^\circ) + x_Z(T;u^{loop},0) + x_Z(T;u^*,0).
    \end{equation}
    By \cref{prop:vertical-loops}, there exists $u^{loop} \in L^\infty((0,T);\R^q)$ such that $x_L(T;u^{loop},0_L) = 0_L$ and 
    \begin{equation}
        x_Z(T;u^{loop},0) = x^*_Z - x_Z(T;u^\circ,x^\circ) - x_Z(T;u^*,0).
    \end{equation}
    We have thus found a control such that $x(3T;u,x^\circ) = x^*$, which proves the result by induction.
\end{proof}

\subsection{Definition of the reduced system}

We describe in this section a ``reduced system'' in which we use the coordinate associated with $p$ as a virtual control for the coordinates along $Z$.
The reduced system is based on the following consequence of \cref{def:Hall} for elements of $Z$.

\begin{definition}
    Let $b \in Z$.
    There exists a unique maximal $m_b \geq 1$ and $\rho(b) \in \B$, which we call the \emph{root} of $b$, such that $b = \ad_p^{m_b} (\rho(b))$.
\end{definition}

\begin{lemma}
    \label{lem:bof-root}
    For all $b \in Z$, one has $p < \rho(b)$, $\Fac(\rho(b)) = \Fac(b) \setminus \{p\} = \{ a \in \Fac(b) \mid a < p \}$.
    Finally $\Fac(\rho(b)) \subset L \cap \B_{<p}$.
\end{lemma}

\begin{proof}
    First, since $b \in \B$ and $m_b \geq 1$, the inner bracket $(p, \rho(b)) \in \B$ so, $p < \rho(b)$ by the second Hall axiom.
    Second, the equality $\Fac(\rho(b)) = \Fac(b) \setminus \{p\}$ follows from the full factorization \eqref{eq:b-factorization} and the maximality of $m_b$.
    Third, by \cref{lem:facto-Hall}, all factors of $\rho(b)$ are strictly smaller than $p$.
    Since they belong to $G$, by \ref{it:bof-G-part-1} of \cref{lem:bof-G-part}, they belong to $L$.
\end{proof}

\begin{lemma}
    \label{lem:rho-b-bad}
    Let $b \in Z$.
    If $\rho(b) \in \B_\bad$, then $m_b$ is odd.
\end{lemma}

\begin{proof}
    Write the full factorization \eqref{eq:b-factorization} of $b$.
    Since $b \in Z \subset G \subset \B_\good$, all its factors are good.
    Thus all factors of $\rho(b)$ are good.
    By \cref{def:of}, if $\rho(b)$ is bad, $X_i = X_0$ and $m_1, \dotsc, m_{r-1}$ are even.
    Since $b$ is good, $m_r = m_b$ is odd.
\end{proof}

The definition \eqref{eq:def:coord2} of $\xi_b$ for $b \in Z$ yields:
\begin{equation}
    \dot{\xi}_b = \frac{\xi_p^{m_b}}{m_b!} \dot{\xi}_{\rho(b)}.
    \label{eq:xib-xip}
\end{equation}
We will need to distinguish good and bad roots.
We define the disjoint sets:
\begin{equation}
    \label{eq:bof-RA}
    \Rg := \rho(Z) \cap \B_\good
    \quad \text{and} \quad
    \Rb := \rho(Z) \cap (\B_\bad \setminus \{ X_0 \}).
\end{equation}
Thus $\Rg \cap \Rb = \varnothing$ and $\Rg \sqcup \Rb \subset \B \setminus \{ X_0 \}$ (we treat the possible root $X_0$ on its own).

To introduce the reduced system, we need the sets:
\begin{equation}
    Y := \{ p \} \sqcup \Rg,
    \quad \text{and} \quad 
    C := Y \sqcup \Rb = \{ p \} \sqcup \Rg \sqcup \Rb.
\end{equation}
The unions are indeed disjoint because any $c \in \Rg \sqcup \Rb$ satisfies $p < c$ by \cref{lem:bof-root}.

\begin{lemma}
    One has $Y \cap Z = \varnothing$.
\end{lemma}

\begin{proof}
    First $p \in L$ (not in $Z$) by \ref{it:bof-G-part-2} of \cref{lem:bof-G-part}.
    Second, any $c \in \Rg$ is of the form $\rho(b)$ for some $b \in Z$, so $p \notin \Fac(\rho(b))$ by  \cref{lem:bof-root}, so $c \notin Z$.
\end{proof}

\begin{definition}[Reduced system]
    Fix $\gamma \in \R^{\Rb}$. 
    We define the \emph{reduced system} associated with~$G$ and $\gamma$ as the following system with states $(y,z) \in \R^Y \times \R^Z$ and controls $v \in \R^C$:
    \begin{equation}
        \label{eq:reduced}
        \dot{y}_b = v_b, 
        \quad \forall b \in Y,
        \qquad \text{and} \qquad
        \dot{z}_b = \frac{y_p^{m_b}}{m_b!} g_b(v),
        \quad \forall b \in Z,
    \end{equation}
    where we set
    \begin{equation}
        g_b(v)
        :=
        \begin{cases}
            1,&\text{if }\rho(b)=X_0,\\
            v_{\rho(b)},&\text{if }\rho(b)\in \Rg,\\
            \gamma_{\rho(b)}+v_{\rho(b)},&\text{if }\rho(b)\in \Rb.
        \end{cases}
        \label{eq:averaged-root-control}
    \end{equation}
\end{definition}

\begin{remark}
    When $\rho(b) \in \Rb$ is a bad root, this reduced system involves a control shifted by $\gamma_{\rho(b)}$ to reflect the fact that the associated coordinate might not be controllable to $0$ (see \cref{lem:normal-box}).
\end{remark}

We rewrite the ODE \eqref{eq:reduced} as a control-affine system.
For each possible root $c \in \Rg \sqcup \Rb \sqcup \{ X_0 \}$, define the vector field
\begin{equation}
    \label{eq:Vc}
    V_c(y_p)
    :=
    \sum_{\substack{b\in Z\\\rho(b)=c}}
    \frac{y_p^{m_b}}{m_b!} \partial_{z_b}.
\end{equation}
Then define the drift as
\begin{equation}
    \label{eq:F0-VC}
    F_0 := F_{X_0} := V_{X_0}+\sum_{c\in \Rb}\gamma_cV_c,
\end{equation}
and the controlled vector fields as
\begin{equation}
    F_p := \partial_{y_p},
    \quad \text{and} \quad
    F_c := \partial_{y_c}+V_c \text{ for } c \in \Rg, 
    \quad \text{and} \quad
    F_c := V_c \text{ for } c \in \Rb.
    \label{eq:FpFc}
\end{equation}
With these notations, \eqref{eq:reduced} can be rephrased as
\begin{equation}
    \label{eq:reduced-F0Fc}
    \frac{\dd}{\dd t} (y,z) = F_0(y_p) + \sum_{c \in C} v_c F_c(y_p).
\end{equation}

We will need the following structural facts in the sequel.

\begin{lemma}
    \label{lem:fac+C}
    One has $Y \cap \B_{<p} = \varnothing$, $\Fac^+(C) \subset L \setminus Y$ and $L \setminus Y$ is factor-stable.
\end{lemma}

\begin{proof}
    First, since $Y = \{ p \} \sqcup \Rg$, $Y \cap \B_{<p} = \varnothing$ because $p \not < p$ and $p < c$ for all $c \in \Rg$ by \cref{lem:bof-root}.
    
    Second, recall that $C = \{ p \} \sqcup \Rg \sqcup \Rb$.
    By \cref{lem:bof-G-part}, $\Fac(p) \subset L \cap \B_{<p}$.
    By \cref{lem:bof-root}, $\Fac(\Rg \sqcup \Rb) \subset L \cap \B_{<p}$.
    By \cref{lem:L}, $L$ is factor-stable.
    By \cref{lem:B<p-stable}, $\B_{<p}$ too.
    Thus $L \cap \B_{<p}$ too.
    By \cref{lem:Fac+ES}, $\Fac^+(C) \subset L \cap \B_{<p}$.
    Finally $L \cap \B_{<p} \subset L \setminus Y$ since $Y \cap \B_{<p} = \varnothing$.

    Third, by \cref{lem:bof-G-part}, $\Fac(L) \subset L \cap \B_{<p}$.
    So $\Fac(L\setminus Y) \subset \Fac(L) \subset L \setminus Y$ since $Y \cap \B_{<p} = \varnothing$.
\end{proof}

\subsection{Controllability of the reduced system}

We investigate the controllability of the reduced system \eqref{eq:reduced-F0Fc}.
We start with computations of the Lie brackets of the involved vector fields.
Let $\mathfrak{L}$ denote the Lie algebra generated by the vector fields in \eqref{eq:F0-VC} and \eqref{eq:FpFc}.

\medskip

The triangular structure of the vector fields $V_c$ of \eqref{eq:Vc} motivates the introduction of $\mathcal{V}$ as the space of polynomial vector fields which depend only on $y_p$ and are supported in the $z$ directions. 
This is an abelian Lie algebra. 
Moreover, for all $V \in \mathcal{V}$,
\begin{equation}
    \label{eq:F0Fc-vertical-ideal}
    [F_c, V] = 0 \quad \text{and} \quad
    [F_c, F_{c'}] = 0, \quad \text{ for all } c, c' \in \Rg \sqcup \Rb \sqcup \{ X_0 \}.
\end{equation}
The only non-trivial brackets are those involving $F_p$.

\begin{lemma}
    \label{lem:reduced-monom-extract}
    For all $c\in \Rg \sqcup \Rb \sqcup\{X_0\}$ and $m \in \N$,
    \begin{equation}
        \bigl(\partial_{y_p}^m V_c\bigr)(0)
        =
        \begin{cases}
            \partial_{z_b},
            &\text{if there exists $b\in Z$ such that }
             \rho(b)=c\text{ and }m_b=m, \\
            0,
            &\text{otherwise}.
        \end{cases}
    \end{equation}
    In particular, the bracket $b$ in the first case is unique.
\end{lemma}

\begin{proof}
    By the definition of $V_c$ and the factorial normalization,
    \begin{equation}
        \bigl(\partial_{y_p}^m V_c\bigr)(0)
        =
        \sum_{\substack{b\in Z\\\rho(b)=c}}
        \mathbf{1}_{\{m_b=m\}} \partial_{z_b}.
    \end{equation}
    Moreover, the pair $(\rho(b),m_b)$ determines $b$ uniquely.
\end{proof}

\begin{proposition}
    \label{prop:reduced-LARC}
    For any $\gamma \in \R^{\Rb}$, the reduced system \eqref{eq:reduced-F0Fc} satisfies the Lie algebra rank condition at the origin, i.e.\ $\mathfrak{L}(0) = \R^Y \times \R^Z$.
\end{proposition}

\begin{proof}
    We obtain the different directions in successive steps.

    \begin{itemize}
        \item 
        First, one has $F_p(0) = \partial_{y_p}$.
        So $\partial_{y_p} \in \mathfrak{L}(0)$.
    
        \item 
        Second, for any $b \in Z$, $m_b \ge 1$, thus, for any $c \in \Rg \sqcup \Rb \sqcup \{ X_0 \}$, $V_c(0) = 0$.
        Hence, for any $a \in \Rg$, $F_a(0) = \partial_{y_a}$. So $\partial_{y_a} \in \mathfrak{L}(0)$ for all $a \in \Rg$.

        \item Third, let $b \in Z$ with a root $c := \rho(b) \in \Rg \sqcup \Rb$ (not equal to $X_0$).
        Then $\ad_{F_p}^{m_b}(F_c) = \partial_{y_p}^{m_b} V_c$, and \cref{lem:reduced-monom-extract} gives $(\partial_{y_p}^{m_b}V_c)(0) = \partial_{z_b}$.
        Hence $\partial_{z_b} \in \mathfrak{L}(0)$.

        \item Fourth, let $b \in Z$ with $\rho(b) = X_0$.
        From \eqref{eq:F0-VC}, we obtain
        \begin{equation}
            \bigl(\partial_{y_p}^{m_b}F_0\bigr)(0)
            =
            \partial_{z_b}
            +
            \sum_{c\in \Rb}
            \gamma_c\bigl(\partial_{y_p}^{m_b}V_c\bigr)(0).
            \label{eq:reduced-drift-root-extraction}
        \end{equation}
        By \cref{lem:reduced-monom-extract}, each term in the sum is either zero or a coordinate direction with a bad root.
        These directions have already been obtained. 
        Hence $\partial_{z_b}\in\mathfrak{L}(0)$.
    \end{itemize}
    We have therefore proved that the canonical basis of $\R^Y \times \R^Z$ is contained in $\mathfrak{L}(0)$.
\end{proof}

We can now prove the main result of this section.

\begin{proposition}
    \label{prop:averaged-system-STLC}
    For any $\gamma\in\R^{\Rb}$, the reduced system \eqref{eq:reduced-F0Fc} is $L^\infty$-STLC at the origin.
\end{proposition}

\begin{proof}
    We plan to apply Sussmann's theorem to the control-affine system \eqref{eq:reduced-F0Fc}.
    By \cref{prop:reduced-LARC}, the Lie algebra rank condition at the origin is satisfied.

    Let $\mathcal{W} := \{ F_0 \} \cup \{ F_c \mid c \in \Rg \sqcup \Rb \}$.
    Any two elements of $\mathcal{W}$ commute.
    Moreover $[\mathcal{W},\mathcal{V}] = 0$ and, for any $W \in \mathcal{W}$, $[F_p, W] = \partial_{y_p} W \in \mathcal{V}$.

    Thus, the only potentially non-zero brackets are the $\ad_{F_p}^m(F_c)$, for $c \in \Rg \sqcup \Rb \sqcup \{X_0\}$.
    Among these, the only Sussmann-bad ones are the $\ad_{F_p}^{2m}(F_0) = \partial_{y_p}^{2m} F_0$ for $m \geq 0$.
    By \eqref{eq:F0-VC} and \eqref{eq:Vc}, the powers of $y_p$ occurring in $F_0$ are the $y_p^{m_b}$ for those $b \in Z$ whose root $\rho(b)$ is bad, and such $m_b$ is odd by \cref{lem:rho-b-bad}.
    So $F_0$ is an odd function of $y_p$, hence so is $\partial_{y_p}^{2m} F_0$, which therefore vanishes at the origin.
    
    Therefore, the conclusion follows from \cref{thm:sussmann} (with any $\theta$).
\end{proof}

Keeping in mind that our goal is to prove \cref{prop:vertical-loops}, i.e.\ to control the components along~$Z$ while bringing the ones along other brackets back to $0$, we will need the following continuous corollary.
The continuous dependence on the target will be required in \cref{subsec:vertical-loops} to absorb the approximation errors which occur in the reduction process.

\begin{corollary}
    \label{lem:reduced-loop}
    Let $\eta > 0$.
    There exists a neighborhood $\Theta$ of $0$ in $\R^Z$ and a continuous map $v : \Theta \to L^\infty((0,1);\R^C)$ such that, for all $\theta \in \Theta$,
    \begin{enumerate}
        \item the control $v^\theta$ takes its values in $[-\eta,\eta]^C$;
        \item the associated trajectory $(y^\theta,z^\theta)$ remains in $[-1,1]^{Y \sqcup Z}$ for $t \in [0,1]$;
        \item its final state is $(y^\theta(1), z^\theta(1)) = (0,\theta)$.
    \end{enumerate}
\end{corollary}

\begin{proof}
    By \cref{prop:reduced-LARC}, the reduced system satisfies the Lie algebra rank condition.
    Thus, by \cref{prop:normal-accessibility}, the reduced system has a normal trajectory in arbitrary state and control neighborhoods in time $\frac 12$.
    Let $(\bar{y},\bar{z})$ be its endpoint. 
    By \cref{prop:averaged-system-STLC}, there is a control driving $(\bar{y},\bar{z})$ to $(0,0)$ in time $\frac 12$.
    If the first trajectory is small, this correction remains in the prescribed neighborhoods.
    Normality is preserved, because the appended flow is a local diffeomorphism.
    The inverse function theorem then gives the result.
\end{proof}

\subsection{Normal boxes around bad brackets}

To perform the induction, we will also need to use \emph{bad} brackets of $\Rb$ as virtual controls.
The difference with good brackets is that the available values for these brackets may not be centered at $0$.
Standard accessibility theory yields the following result.

\begin{lemma}[Normal box]
    \label{lem:normal-box}
    Let $H\subset\B\setminus\{X_0\}$ be finite and factor-stable. 
    Set $H_\good := H \cap \B_\good$ and $H_\bad := H \cap \B_\bad$.
    Assume that $\Sigma_{H_\good}$ is small-time globally controllable.
    There exist $\gamma\in\R^{H_\bad}$, a neighborhood $\Omega_H$ of $0$ in $\R^H$, and a $C^1$ map $\bar{u} : \Omega_H \to L^\infty((0,1);\R^q)$ such that, for all $w \in \Omega_H$,
    \begin{alignat}{2}
        & \forall b \in H_\good, & \qquad \xi_b(1,\bar{u}[w]) & = w_b, \\
        & \forall b \in H_\bad, & \qquad \xi_b(1,\bar{u}[w]) & = \gamma_b + w_b.
    \end{alignat}
\end{lemma}

\begin{proof}
    By \eqref{eq:Hall-H-Ev}, $\Sigma_H$ satisfies the Lie algebra rank condition.
    Apply \cref{prop:normal-accessibility} to $\Sigma_H$ on the first half of the time interval. 
    Let $(x_\good,x_\bad)$ be the endpoint of the resulting normal control. 
    Since both $H$ and $\B_\good$ are factor-stable, their intersection is too. 
    Hence the subsystem indexed by $H_\good$ is closed.
    On the second half of the interval, append a fixed control which drives $x_\good$ back to~$0$. 
    The endpoint of the full system is then of the form $(0,\gamma)$.

    The flow associated with the appended control is a local diffeomorphism.
    Hence the concatenated endpoint map is still submersive. Choose finitely many directions in $L^\infty$ on which its differential is an isomorphism. 
    The inverse function theorem gives the conclusion.
\end{proof}

\subsection{Approximation argument}
\label{subsec:averaging}

Let $\kappa \geq 1$.
We explain the link between the reduced system and the initial one.

\paragraph{Application of the induction hypothesis.}
Since we proceed by induction on the number of distinct factors $|\Fac(G)|$, we assume that \cref{thm:bss-good} holds for all $G$ such that $|\Fac(G)| < \kappa$.

Let $G \subset \B_\good$ be finite and factor-stable with $|\Fac(G)| = \kappa$.
Consider the finite subset $H := L \cup \Rg \cup \Rb$ of $\B \setminus \{ X_0 \}$. 
By \cref{lem:bof-G-part,lem:fac+C}, $\Fac(H) \subset L \setminus Y \subset H$, so $H$ is factor-stable.
Let $H_\good := H \cap \B_\good = L \cup \Rg = (L \setminus Y) \sqcup Y$.
By the same lemmas, $\Fac(H_\good) \subset \Fac(G) \setminus \{ p \}$, so $|\Fac(H_\good)| < \kappa$.
By the induction hypothesis, $\Sigma_{H_\good}$ is small-time globally controllable.
Apply \cref{lem:normal-box} to $H$. 
Up to restricting the resulting map, we obtain $\gamma\in\R^{\Rb}$, $\eta > 0$ and a $C^1$ map $\bar{u} : [-\eta,\eta]^C \to L^\infty((0,1);\R^q)$ such that, for all $w \in [-\eta,\eta]^C$,
\begin{alignat}{2}
    \label{eq:baru-g0}
    &\forall b \in L \setminus Y,  &\qquad \xi_b(1,\bar{u}[w]) &= 0, \\
    \label{eq:baru-gwa}
    &\forall b \in Y, &\qquad \xi_b(1,\bar{u}[w]) &= w_b, \\
    \label{eq:baru-gwc}
    &\forall b \in \Rb,       &\qquad \xi_b(1,\bar{u}[w]) &= \gamma_b + w_b.
\end{alignat}
From now on, $\gamma \in \R^{\Rb}$, $\eta > 0$ and the $C^1$ map $\bar{u}$ are fixed.

\paragraph{Controllability of the reduced system.}
By \cref{prop:averaged-system-STLC}, the reduced system \eqref{eq:reduced} associated with $\gamma$ is $L^\infty$-STLC.
More precisely, by \cref{lem:reduced-loop}, there exist a neighborhood $\Theta$ of~$0$ in $\R^Z$ and a continuous map $v : \Theta \to L^\infty((0,1);\R^C)$ such that, for all $\theta \in \Theta$, $v^\theta$ takes values in~$[-\eta,\eta]^C$, the associated trajectories $(y^\theta,z^\theta)$ remain in $[-1,1]^{Y \sqcup Z}$, and the final state satisfies $(y^\theta,z^\theta)(1) = (0,\theta)$. 

\paragraph{Construction of approximate trajectories.}
Let $\theta \in \Theta$.
We construct a sequence of trajectories of the system $\Sigma_G$ approximating the trajectory $(y^\theta,z^\theta)$ of the reduced system driven by~$v^\theta$.

Let $N \geq 1$.
For $0 \leq k < N$, set $I_{N,k}:=[\frac{k}{N},\frac{k+1}{N})$ and
\begin{equation}
    \label{eq:wNktheta}
    w^{\theta,N,k}
    :=
    N\int_{I_{N,k}}v^\theta(t)\dd t
    \in [-\eta,\eta]^C,
\end{equation}
using the fact that $v^\theta$ takes values in $[-\eta,\eta]^C$.

We define a control $u^{\theta,N} \in L^\infty((0,N);\R^q)$ by concatenation as follows:
\begin{equation}
    u^{\theta,N} := \bar{u}[w^{\theta,N,0}] \diamond \dotsb \diamond \bar{u}[w^{\theta,N,N-1}].
\end{equation}
By \cref{lem:fac+C}, $L\setminus Y$ is factor-stable.
Thus \cref{lem:concatenation-factor-null} and \eqref{eq:baru-g0} entail that
\begin{equation}
    \label{eq:xi-a-Glow-0}
    \forall b \in L \setminus Y, 
    \quad 
    \forall 0 \leq k \leq N, 
    \qquad
    \xi_b(k, u^{\theta,N}) = 0
\end{equation}
By \cref{lem:fac+C}, $\Fac(Y \sqcup \Rb) \subset L \setminus Y$.
Thus \cref{lem:concatenation-factor-null} and \eqref{eq:xi-a-Glow-0} entail that
\begin{equation}
    \forall k \in \intset{0,N-1}, \forall s \in [0,1],
    \forall b \in Y \sqcup \Rb \cup \{ X_0 \},
    \quad
    \dot{\xi}_b(k+s,u^{\theta,N}) = \dot{\xi}_b(s,\bar{u}[w^{\theta,N,k}]).
\end{equation}
Thus, \eqref{eq:baru-gwa} and \eqref{eq:wNktheta} entail that:
\begin{equation}
    \forall b \in Y, \forall 0 \leq k \leq N, \qquad
    \xi_b(k, u^{\theta,N}) 
    = N \int_0^{\frac{k}{N}} v^\theta_b 
    = N y^\theta_b \left(\frac{k}{N}\right).
\end{equation}
In particular, one has $\xi_b(N,u^{\theta,N}) = 0$ for $b \in L \setminus Y$ and $\xi_b(N, u^{\theta,N}) = N y^\theta_b(1) = 0$ for $b \in Y$.

\paragraph{Convergence of the approximation.}
Fix $b\in Z$.
We now prove that, uniformly for $\theta \in \Theta$,
\begin{equation}
    \frac{\xi_b(N,u^{\theta,N})}{N^{m_b+1}}
    =
    z_b^\theta(1)+O\left(\frac1N\right).
    \label{eq:higher-depth-convergence}
\end{equation}
Since $[-\eta,\eta]^C$ is compact and $\bar{u}$ is continuous on $[-\eta,\eta]^C$ with values in $L^\infty$, there exists $M > 0$ such that, for all $a \in Y \sqcup \Rb \sqcup \{ X_0 \}$, and all $w \in [-\eta,\eta]^C$,
\begin{equation}
    \label{eq:block-bounds}
    \sup_{s \in [0,1]} | \xi_a(s,\bar{u}[w]) | \leq M
    \quad \text{and} \quad
    \int_0^1 | \dot{\xi}_a(s,\bar{u}[w]) | \dd s \leq M.
\end{equation}
To lighten the computations, we also set $w^k:=w^{\theta,N,k}$.
During the $k$-th block,
\begin{equation}
    \xi_b(k+1,u^{\theta,N})-\xi_b(k,u^{\theta,N})
    =
    \frac{1}{m_b!}
    \int_0^1
    \left(N y^\theta_p\left(\frac{k}{N}\right)+\xi_p(s,\bar{u}[w^k])\right)^{m_b}
    \dot{\xi}_{\rho(b)}(s,\bar{u}[w^k])\dd s.
\end{equation}
By the binomial theorem, estimates \eqref{eq:block-bounds} yield
\begin{equation}
    \left(Ny^\theta_p\left(\frac{k}{N}\right)+\xi_p(s,\bar{u}[w^k])\right)^{m_b}
    =
    N^{m_b} \left(y^\theta_p\left(\frac{k}{N}\right)\right)^{m_b} +O(N^{m_b-1}).
\end{equation}
Moreover, recalling \eqref{eq:averaged-root-control}, the identities \eqref{eq:baru-gwa} and \eqref{eq:baru-gwc} imply that, for all $w \in [-\eta,\eta]^C$,
\begin{equation}
    \int_0^1\dot{\xi}_{\rho(b)}(s,\bar{u}[w])\dd s
    = \xi_{\rho(b)}(1, \bar{u}[w]) = g_b(w).
\end{equation}
Consequently,
\begin{equation}
    \frac{
    \xi_b(k+1,u^{\theta,N})-\xi_b(k,u^{\theta,N})
    }{N^{m_b+1}}
    =
    \frac1N \frac{(y^\theta_p(\frac{k}{N}))^{m_b}} {m_b!} g_b(w^k)
    +O\left(\frac1{N^2}\right).
\end{equation}
Moreover, since $g_b$ is affine, by the definition of $w^k = w^{\theta,N,k}$ in \eqref{eq:wNktheta},
\begin{equation}
    g_b(w^k)
    =
    N\int_{I_{N,k}}g_b(v^\theta(t))\dd t.
\end{equation}
Summing the block increments therefore gives
\begin{equation}
    \frac{\xi_b(N,u^{\theta,N})}{N^{m_b+1}}
    =
    \sum_{k=0}^{N-1}
    \int_{I_{N,k}}
    \frac{(y^\theta_p(\frac{k}{N}))^{m_b}} {m_b!} g_b(v^\theta(t))\dd t
    +O\left(\frac1N\right).
\end{equation}
Finally, the trajectories $y^\theta$ are uniformly Lipschitz, so $y^\theta_p(t)=y^\theta_p(\frac{k}{N})+O(N^{-1})$ for $t\in I_{N,k}$.
Since $g_b(v^\theta)$ is uniformly bounded, the preceding sum equals
\begin{equation}
    \int_0^1
    \frac{(y^\theta_p(t))^{m_b}} {m_b!}g_b(v^\theta(t))\dd t
    +O\left(\frac1N\right).
\end{equation}
The integral is $z_b^\theta(1) = \theta_b$ by the reduced equation, proving
\eqref{eq:higher-depth-convergence}.

\bigskip
We have proved the following statement.

\begin{proposition}
    \label{lem:C0-averaging-d>1}
    For each $N \ge 1$, there exists a continuous map $\theta \mapsto u^{\theta,N}$ from the neighborhood $\Theta \subset \R^Z$ to $L^\infty((0,N);\R^q)$ such that, uniformly for $\theta \in \Theta$,
    \begin{equation}
        \forall b \in G \setminus Z, \enskip
        \xi_b(N,u^{\theta,N}) = 0
        \qquad \text{and} \qquad
        \forall b \in Z, \enskip
        \frac{\xi_b(N,u^{\theta,N})}{N^{m_b+1}} = \theta_b + O(N^{-1}).
    \end{equation}
\end{proposition}

\subsection{From uniform approximation to control loops}
\label{subsec:vertical-loops}

We now use the uniform approximation obtained in \cref{lem:C0-averaging-d>1} to prove \cref{prop:vertical-loops}.
We first recall a standard topological argument.

\begin{lemma}
    \label{lem:C0-stability-surjectivity}
    Let $\Theta\subset\R^n$ be a neighborhood of $0$ and $F_N\in C^0(\Theta;\R^n)$ converging uniformly to the identity map.
    Then there exist $\varepsilon>0$ and $N\in\N$ such that $B(0,\varepsilon) \subset F_N(\Theta)$.
\end{lemma}

\begin{proof}
    Let $\varepsilon > 0$ such that $\overline{B}(0,2\varepsilon) \subset \Theta$.
    For $N$ large enough, $|F_N(y) - y| \leq \varepsilon$ for all $y \in \Theta$.
    Fix $z\in B(0,\varepsilon)$ and define $J_{N,z}(y):=z+y-F_N(y)$. 
    This continuous map sends $\overline{B}(0,2\varepsilon)$ into itself. 
    Brouwer's fixed-point theorem yields $y \in \Theta$ such that $J_{N,z}(y)=y$, and therefore $F_N(y)=z$.
\end{proof}

We now conclude the proof.

\begin{proof}[Proof of \cref{prop:vertical-loops}]
    Let $\Xi_G^+(T)$ denote the set on the left-hand side of
    \eqref{eq:vertical-loops}.

    \medskip
    \noindent \emph{Step 1: We prove that there exists $N \geq 1$ and $\varepsilon > 0$ such that $B(0,\varepsilon) \subset \Xi_G^+(N)$.}
    
    For $N\geq 1$, define $F_N : \Theta \to \R^Z$ by
    \begin{equation}
        F_N(\theta) := \left(\frac{\xi_b(N,u^{\theta,N})}{N^{m_b+1}}\right)_{b\in Z}.
    \end{equation}
    By \cref{lem:C0-averaging-d>1}, the maps $F_N$ are continuous and converge uniformly to the identity on $\Theta$.
    Hence \cref{lem:C0-stability-surjectivity} provides $\varepsilon>0$ and $N\geq 1$ such that $B(0,\varepsilon) \subset F_N(\Theta)$.    Since each control $u^{\theta,N}$ satisfies
    \begin{equation}
        \xi_a(N,u^{\theta,N})=0
        \qquad
        \text{for every }a\in G\setminus Z,
    \end{equation}
    it follows that
    \begin{equation}
        \left\{
            \left(N^{m_b+1}z_b\right)_{b\in Z}
            \mid
            z\in B_{\R^Z}(0,\varepsilon)
        \right\}
        \subset \Xi_G^+(N).
    \end{equation}
    In particular, because $N^{m_b+1}\geq 1$ for every $b\in Z$, $B(0,\varepsilon) \subset \Xi_G^+(N)$.
    
    \medskip
    \noindent \emph{Step 2: We next prove that $\Xi_G^+(N)=\R^Z$.}
    For $b \in \Br(X)$, set $n(b) := n_1(b) + \dotsb + n_q(b)$.
    For $\lambda>0$ and $z \in \R^Z$, set $\Lambda_\lambda(z):=\left(\lambda^{n(b)}z_b\right)_{b\in Z}$.
    Applying \eqref{eq:xi-control-homogeneity} with
    $\lambda_1=\dotsb=\lambda_q=\lambda$ shows that $\Lambda_\lambda\bigl(\Xi_G^+(N)\bigr) \subset \Xi_G^+(N)$.
    Moreover, $n(b)\geq 1$ for every $b\in Z$, since $Z\subset G\subset\B_\good$ and $X_0\notin\B_\good$.
    Therefore, for any $z\in\R^Z$, one can choose $\lambda \geq 1$ sufficiently large that $\Lambda_{1/\lambda}(z) \in B(0,\varepsilon)$.
    Since $B(0,\varepsilon)\subset\Xi_G^+(N)$, we obtain $z=\Lambda_\lambda\bigl(\Lambda_{1/\lambda}(z)\bigr)\in\Xi_G^+(N)$.
    Thus $\Xi_G^+(N)=\R^Z$.
    
    \medskip
    \noindent \emph{Step 3: We conclude that, for any $T > 0$, $\Xi_G^+(T)=\R^Z$.}
    Let $T > 0$ and $z\in\R^Z$.
    Since $\Xi_G^+(N)=\R^Z$, there exists a control $u\in L^\infty((0,N);\R^q)$ such that
    \begin{equation}
        \xi_b(N,u)=0
        \quad\text{for every }b\in G\setminus Z,
        \qquad
        \xi_b(N,u)
        =
        \left(\frac{N}{T}\right)^{n_0(b)}z_b
        \quad\text{for every }b\in Z.
    \end{equation}
    Define $u^{N,T}:=\frac{N}{T}u\left(\frac{N}{T}\cdot\right)$.
    By \eqref{eq:xi-time-homogeneity},
    \begin{equation}
        \xi_b(T,u^{N,T})=0
        \quad\text{for every }b\in G\setminus Z,
        \qquad
        \xi_b(T,u^{N,T})=z_b
        \quad\text{for every }b\in Z.
    \end{equation}
    This concludes the proof.
\end{proof}

\newpage

\section{Proof of the necessary condition}
\label{sec:observability}

We prove \cref{thm:necessary} by reducing it to the following purely algebraic ``observability result'', which does not involve vector fields, or controls.
The idea is that, under the conditions of \cref{thm:necessary}, one can find an \emph{embedded system} within \eqref{eq:syst} which is not $L^\infty$-STLC.
This embedded system is not STLC because one of its coordinates is $\xi_b$, the coordinate of the second kind associated with the isolated bad bracket $b$, which satisfies $\xi_b(t,u) \ge 0$.

\begin{lemma}
    \label{lem:xibad>0}
    Let $\B$ be a factor-parity Hall set on $X = X_\good \sqcup X_\bad$ where $X_\good = \{ X_1, \dotsc, X_q \}$ and $X_\bad = \{ X_0 \}$.
    For any $b \in \B_\bad$, $T > 0$, $u \in L^1((0,T);\R^q)$ and $t \in [0,T]$,
    \begin{equation}
        \xi_b(t,u) \ge 0.
    \end{equation}
\end{lemma}

\begin{proof}
    We prove by induction on $|b|$ the stronger statement that $\dot{\xi}_b(\cdot,u) \ge 0$ a.e.\ on $(0,T)$ for all $b \in \B_\bad$.
    Since $\xi_b(0,u)=0$, this implies the desired conclusion.

    For $|b|=1$, one has $b=X_0$, hence $\dot{\xi}_b = 1$.
    Let now $b\in\B_\bad$ with $|b|>1$, and write
    \begin{equation}
        b
        =\ad_{a_r}^{m_r}\dotsb\ad_{a_1}^{m_1}(\seed(b))
        \quad \text{so that} \quad
        b = \ad_a^m(c)
    \end{equation}
    where $a := a_r$, $m := m_r$ and $c:=\ad_{a_{r-1}}^{m_{r-1}}\dotsb \ad_{a_1}^{m_1}(\seed(b))$ (with $c=\seed(b)$ when $r=1$).

    We claim that $c\in\B_\bad$ and that either $a\in\B_\bad$ or $m$ is even.
    Indeed, if $a\in\B_\bad$, then $a<c$ and $\B_\good<\B_\bad$, so necessarily $c\in\B_\bad$.
    If $a\in\B_\good$, then $a_1 < \dotsb < a_r = a$ implies $a_1,\dotsc,a_r\in\B_\good$, since $\B_\good<\B_\bad$.
    As $b\in\B_\bad$, the factor-parity property therefore forces $\seed(b) = X_0$ and $m_1, \dotsc, m_r$ to be even.
    In particular $m$ is even, and \eqref{eq:of} shows that $c\in\B_\bad$.

    By \eqref{eq:def:coord2},
    \begin{equation}
        \dot\xi_b
        =\frac{\xi_a^m}{m!}\dot\xi_c.
        \label{eq:bad-xi-induction}
    \end{equation}
    By the induction hypothesis, $\dot\xi_c\geq0$ a.e. Moreover, if $a\in\B_\bad$, then the induction hypothesis also gives $\xi_a\geq0$; while if $a\in\B_\good$, $m$ is even. 
    In either case, $\xi_a^m \ge 0$.
    Hence \eqref{eq:bad-xi-induction} yields $\dot\xi_b\geq0$ a.e. 
\end{proof}

\begin{proposition}[Observability result]
    \label{prop:observability}
    Let $\B$ be a factor-parity Hall set on $X = X_\good \sqcup X_\bad$.
    Let $K \subset \mathcal{L}(X)$ be a Lie subalgebra and $b \in \B_\bad$ satisfy
    \begin{equation}
        \label{eq:no-other-bad}
        \B_\bad\setminus\{b\}\subset K
    \end{equation}
    and
    \begin{equation}
        \label{eq:b-notin-supp-K}
        b\notin\supp_{\B}(K).
    \end{equation}
    Then
    \begin{equation}
        \label{eq:fac*b-notin-supp-K}
        \Fac^*(b)\cap\supp_{\B}(K)=\varnothing.
    \end{equation}
\end{proposition}

\begin{proof}[Proof of \cref{thm:necessary}]
    Let $\B$ be a factor-parity Hall set on $X = X_\good \sqcup X_\bad$ where $X_\good = \{ X_1, \dotsc, X_q \}$ and $X_\bad = \{ X_0 \}$.
    Let $K_f \subset \mathcal{L}(X)$ be the kernel of the linear map $B \mapsto f_B(0) \in \R^d$.
    Since the Lie bracket of two vector fields vanishing at $0$ vanishes at~$0$, $K_f$ is a Lie subalgebra of $\mathcal{L}(X)$.
    For all $a \in \B_\bad \setminus \{b\}$, $f_a(0) = 0$ by \eqref{eq:syst-no-other-bad} so $a \in K_f$ and $K_f$ satisfies \eqref{eq:no-other-bad}.
    Moreover $K_f$ also satisfies~\eqref{eq:b-notin-supp-K}.
    Otherwise, there would exist $z \in K_f$ such that $b \in \supp_\B (z)$.
    Up to rescaling, one can assume that $\pair{z}{b}_\B = 1$.
    Hence $z = b + \sum_{a \in \B \setminus \{ b \}} \lambda_a a$ with a finite sum.
    Thus $0 = f_z(0) = f_b(0) + \sum_{a \in \B \setminus \{b \}} \lambda_a f_a(0)$ and $f_b(0) \in \vect \{ f_a(0) \mid a \in \B, \ a \neq b \}$, contradicting \eqref{eq:syst-b-notin-supp-K}.
    Thus \cref{prop:observability} applies and proves that $K_f \subset V$ where $V := \vect (\B \setminus \Fac^*(b)) \subset \mathcal{L}(X)$.
    
    By \cite[Lemma 5.32]{BeauchardLeBorgneMarbach2026}, $V$ is a Lie subalgebra of $\mathcal{L}(X)$ because $H := \Fac^*(b)$ is factor-stable.
    Let $\Sigma_H$ be the canonical system of the coordinates of the second kind associated with $H$, as in \cref{def:canonical}.
    By \cite[Proposition 5.7]{BeauchardLeBorgneMarbach2026}, $\Sigma_H$ is embedded in \eqref{eq:syst}, i.e.\ there exists a smooth local submersion $\theta : \R^d \to \R^H$ (with $\theta(0) = 0$ and $D\theta(0)$ onto) such that $y(t;u) = \theta(x(t;u))$, where $y$ is the solution to $\Sigma_H$ and $x$ the solution to~\eqref{eq:syst}.
    By \cref{lem:xibad>0}, $\xi_b(t,u) \ge 0$.
    Since $y_b(t;u) = \xi_b(t,u)$, the system $\Sigma_H$ is not $L^\infty$-STLC, and hence neither is \eqref{eq:syst}. 
\end{proof}

\subsection{Strategy for the proof of the observability result}

We plan to prove \cref{prop:observability} by induction on $n:=\abs{\Fac^*(b)}$, its statement being understood as universally quantified over all data $(X,\B,K,b)$.
The case $n=1$ is \eqref{eq:b-notin-supp-K}.
In the induction step, one first discards the letters not occurring in $b$ (\cref{lem:restriction}), so that $p:=\min\Fac^*(b)$ becomes the minimal letter of the alphabet, and is a good one (\cref{sec:irreducible}). 
Lazard elimination of~$p$ then manufactures a \emph{new} instance $(Y,\B_Y,K_Y,\hat{b})$ of the same statement, with $\abs{\Fac^*_Y(\hat{b})}=n-1$ (\cref{lem:lazard}).
Applying the induction hypothesis to it excludes $\Fac^*(b)\setminus\{p\}$ from $\supp_{\B}(K\cap\Lie(Y))$. 
It remains to exclude $p$ itself: if some $h\in K$ had $\pair{h}{p}_{\B}=1$, \cref{thm:support} below would produce $c\in\Fac^*(b)\setminus\{p\}$ detected by $\Lie(\{p\}\sqcup(\B_\bad\setminus\{b\}))$, and substituting $h$ for $p$ would move that detection inside $K\cap\Lie(Y)$, a contradiction.

The following technical result will be proved in \cref{sec:support}.

\begin{proposition}[Support result]
    \label{thm:support}
    Let $\B$ be a factor-parity Hall set on $X=X_\good\sqcup X_\bad$. Let $b\in\B_\bad\setminus X$ with $\Fac(b)\subset\B_\good$, and assume $p:=\min\Fac^*(b)\in X_\good$.
    Then
    \begin{equation}
        \label{eq:fac*b+p-cap}
        \bigl(\Fac^*(b)\setminus\{p\}\bigr)
        \cap
        \supp_{\B}\Lie\bigl(\{p\}\sqcup(\B_\bad\setminus\{b\})\bigr)
        \neq\varnothing .
    \end{equation}
\end{proposition}

Let us give examples to illustrate, in some particular cases, that this result holds.
Consider $X = \{ X_0, X_1, X_2 \}$ with $X_\good = \{ X_1, X_2 \}$ and $X_\bad = \{ X_0 \}$.
Let $\B$ be a factor-parity Hall set on $X$ whose order refines length within $\B_\good$ and $\B_\bad$.
Assume that $X_1 < X_2$.
\begin{itemize}
    \item \textbf{Case} $b = \ad_{X_1}^2(X_0)$.
    Then $\Fac^*(b) = \{ b, X_1 \}$ and $p = X_1$.
    Since $X_0 \in \B_\bad \setminus \{ b \}$, one has $\Lie(\{p\} \cup (\B_\bad \setminus \{b\})) \supset \Lie(\{X_0, X_1\})$.
    Thus $b$ belongs to the intersection in \eqref{eq:fac*b+p-cap}.
    More generally, this works as soon as $b \in \Lie(\{X_0,X_1\})$.
    
    \item \textbf{Case} $b = \ad_{(X_1,X_2)}^2(X_0)$.
    Then $\Fac^*(b) = \{ b, (X_1, X_2), X_1 \}$ and $p = X_1$.
    Then $\ad_{X_2}^2(X_0) \in \B_\bad \setminus \{ b \}$, so $z := \ad_{X_1}^2 \ad_{X_2}^2(X_0) \in \Lie(\{p\}\sqcup(\B_\bad\setminus\{b\}))$.
    Using the Jacobi identity, one obtains
    \begin{equation}
        \begin{aligned}
            z=
            2b&+\ad_{X_2}^2\ad_{X_1}^2(X_0)
            +4\bigl[[X_1,X_2],[X_2,[X_1,X_0]]\bigr] \\
            &-2\bigl[[X_1,X_0],[X_2,[X_1,X_2]]\bigr]
            +\bigl[[X_2,[X_1,[X_1,X_2]]],X_0\bigr] \\
            &-2\bigl[[X_2,X_0],[X_1,[X_1,X_2]]\bigr].
        \end{aligned}
    \end{equation}
    Every bracket on the right-hand side belongs to $\B$. 
    Thus this is the expansion of $z$ on $\B$ and $\pair{z}{b}_\B = 2$, which proves \eqref{eq:fac*b+p-cap} in this case.
\end{itemize}

\subsection{Irreducible bad brackets}
\label{sec:irreducible}

We start with an elementary observation on the structure of brackets satisfying \eqref{eq:no-other-bad} and \eqref{eq:b-notin-supp-K}.

\begin{lemma}
    \label{lem:irreducible}
    Let $\B$ be a factor-parity Hall set on $X=X_\good\sqcup X_\bad$.
    Let $K\subset\mathcal{L}(X)$ be a Lie subalgebra and $b\in\B_\bad$ satisfy \eqref{eq:no-other-bad} and \eqref{eq:b-notin-supp-K}.
    Then $\Fac(b) \subset \B_\good$.

    If moreover $\abs{\Fac^*(b)} > 1$, then $b \notin X$ and $p := \min \Fac^*(b) \in X_\good$.
\end{lemma}

\begin{proof}
    Assume some $a_j\in\Fac(b)$ is bad.
    Since $a_1 < \dotsb < a_r$ and $\B_\good<\B_\bad$, $a_r \in \B_\bad$.
    Put
    \begin{equation}
        b':=\ad_{a_r}^{m_r-1}\ad_{a_{r-1}}^{m_{r-1}}\cdots\ad_{a_1}^{m_1}(\seed(b))\in\B,
        \qquad\text{so that}\qquad
        b=(a_r,b').
    \end{equation}
    By the second Hall axiom, $a_r < b'$, hence $b' \in \B_\bad$.
    Since $\abs{b} = \abs{a_r} + \abs{b'}$, both $a_r$ and $b'$ belong to $\B_\bad\setminus\{b\} \subset K$ by \eqref{eq:no-other-bad}.
    As $K$ is a Lie subalgebra, $b = (a_r,b') \in K$, contradicting \eqref{eq:b-notin-supp-K}.
    Hence $\Fac(b)\subset\B_\good$, and factor-stability of $\B_\good$ gives $\Fac^+(b)\subset\B_\good$, i.e.\ $\Fac^*(b)\cap\B_\bad=\{b\}$.
    By~\eqref{eq:of}, since $\Fac(b) \subset \B_\good$ and $b \in \B_\bad$, $\seed(b) \notin X_\good$ and no $m_j$ is odd.

    If $b \in X$, $\Fac^*(b) = \{ b \}$.
    Hence $\abs{\Fac^*(b)} > 1$ entails that $b \notin X$, while $p \in X$ by \cref{lem:min-fac}~\ref{it:min-fac-2}.
    Hence $p \neq b$ and $p \in \Fac^+(b)\subset\B_\good$, so $p\in\B_\good\cap X=X_\good$ by \cref{rem:of}.
\end{proof}

\subsection{Two elementary reductions}\label{sec:reductions}

The induction of \cref{sec:proof-main} produces, from a datum $(X,\B,K,b)$, a new datum of exactly the same nature, over a new alphabet. 
The two lemmas of this subsection manufacture it.

\begin{lemma}[Restriction]
    \label{lem:restriction}
    Let $\B$ be a factor-parity Hall set on $X=X_\good\sqcup X_\bad$ and let $S\subset X$.
    Set $\B_S:=\B\cap\Br(S)$, with the induced order. Then
    \begin{enumerate}[label=\textup{(\roman*)},leftmargin=2em]
        \item\label{it:restr-hall} $\B_S$ is a Hall set on $S$, and $\B_S=\{c\in\B \mid \supp\deg_X(c)\subset S\}$;
        \item\label{it:restr-fac} for $c\in\B_S$, the factorizations of $c$ on $S$ and on $X$ coincide; in particular $\Fac^*_S(c)=\Fac^*(c)$;
        \item\label{it:restr-of} $\B_S$ is of the factor-parity class for $S=(S\cap X_\good)\sqcup(S\cap X_\bad)$, with $(\B_S)_\bad=\B_\bad\cap\Br(S)$;
        \item\label{it:restr-rho} the Lie morphism $\rho_S:\mathcal{L}(X)\to\mathcal{L}(S)$ determined by $\rho_S|_S=\Id$ and $\rho_S|_{X\setminus S}=0$ satisfies
        \begin{equation}
            \label{eq:retraction}
            \pair{\rho_S(v)}{c}_{\B_S}=\pair{v}{c}_{\B}
            \quad \text{for all} \quad
            v\in\mathcal{L}(X),\ c\in\B_S.
        \end{equation}
    \end{enumerate}
\end{lemma}

\begin{proof}
    \ref{it:restr-hall} For $c\in\Br(X)$, $c\in\Br(S)$ iff $\supp\deg_X(c)\subset S$. 
    The three axioms of \cref{def:Hall} for~$\B_S$ follow from those for $\B$, using $X\cap\Br(S)=S$.

    \ref{it:restr-fac} In \eqref{eq:b-factorization} for $c\in\B_S$, all $a_j$ and $\seed(c)$ lie in $\Br(S)$, hence in $\B_S$; uniqueness in \cref{lem:factorization} identifies the two factorizations, and \eqref{eq:nested} gives $\Fac^*_S(c)=\Fac^*(c)$ by induction on $\abs{c}$.

    \ref{it:restr-of} The order is induced, so $\B_\good\cap\Br(S)<\B_\bad\cap\Br(S)$. 
    By \ref{it:restr-fac}, the right-hand side of \eqref{eq:of} is the same on $S$ and on $X$, whence the statuses agree by induction on $\abs{c}$.

    \ref{it:restr-rho} 
    By linearity $\rho_S(v) = \sum_{a \in \B} \pair{v}{a}_\B \rho_S(a)$ and $\pair{\rho_S(a)}{c}_{\B_S} = \delta_{a,c}$.
\end{proof}

The following lemma, which states that Lazard elimination preserves factor-parity Hall sets, is a key argument of our inductive proof.

\begin{lemma}[Lazard elimination]
    \label{lem:lazard}
    Let $\B$ be a factor-parity Hall set on $X=X_\good\sqcup X_\bad$ and $p := \min X \in X_\good$.
    Define a new alphabet $Y = Y_\good \sqcup Y_\bad$ as the set of pairs
    \begin{equation}
        Y:= \big\{ y[n] : y\in X\setminus\{p\},\ n\in\N \big\},
        \quad \text{and} \quad
        Y_\bad := \big\{ y[n] : y \in X_\bad \text{ and $n$ even} \big\}.
    \end{equation}
    Let $\Psi:\Br(Y)\to\Br(X)$ be the magma morphism with $\Psi(y[n]):=\ad_p^n(y)$, $\Phi:\mathcal{L}(Y)\to\mathcal{L}(X)$ be the Lie morphism with $\Phi(y[n]):=\eval\ad_p^n(y)$.
    Set $\B_Y:=\Psi^{-1}(\B)$, ordered by $c<c'\iff\Psi(c)<\Psi(c')$. 
    Then
    \begin{enumerate}[label=\textup{(\roman*)},leftmargin=2em]
        \item\label{it:laz-hall} $\Psi$ is injective, $\Psi(\B_Y)=\B\setminus\{p\}$, and $\B_Y$ is a Hall set on $Y$;
        \item\label{it:laz-alg} $\Phi$ is injective, $\Phi\circ\eval=\eval\circ\Psi$, and $I:=\Phi(\mathcal{L}(Y))$ is an ideal of $\mathcal{L}(X)$ with $\mathcal{L}(X)=\R p\oplus I$;
        \item\label{it:laz-coord} $\eval(\B\setminus\{p\})$ is a basis of $I$ and
        \begin{equation}
            \label{eq:coord-transport}
            \pair{\Phi(v)}{\Psi(c)}_{\B}=\pair{v}{c}_{\B_Y}
            \quad \text{for all} \quad v\in\mathcal{L}(Y),\ c\in\B_Y;
        \end{equation}
        \item\label{it:laz-of} $\B_Y$ is of the factor-parity class for $Y=Y_\good\sqcup Y_\bad$, and $\Psi\bigl((\B_Y)_\bad\bigr)=\B_\bad$;
        \item\label{it:laz-fac} $\Psi\bigl(\Fac^*_Y(c)\bigr)=\Fac^*(\Psi(c))\setminus\{p\}$ for every $c\in\B_Y$.
    \end{enumerate}
\end{lemma}

\begin{proof}
    \ref{it:laz-hall} is \cite[Lemma 4.19]{Reutenauer1993} (with reversed Hall set conventions).
    Note that $p\notin\Psi(\Br(Y))$, since each element of $\Psi(\Br(Y))$ involves a letter different from $p$.
    Injectivity of $\Psi$ and $\Psi(\Br(Y))\cap\B=\B\setminus\{p\}$ follow by induction on length, using $\min\B=\min X=p$ and \eqref{eq:b-factorization}.

    \ref{it:laz-alg} This is the classical Lazard elimination, see \cite[Section 0.3]{Reutenauer1993} or \cite[Chapter 1, Section 2]{Viennot1978}.
    
    \ref{it:laz-coord} For $c\in\B\setminus\{p\}$, $\deg_X(c) \notin \N e_p$, so $\eval(c)\in I$. 
    As $\eval(\B)$ is a basis of $\mathcal{L}(X)=\R p\oplus I$, $\eval(\B\setminus\{p\})$ is a basis of $I$. 
    By \ref{it:laz-hall}--\ref{it:laz-alg}, $\Phi$ maps the basis $\eval(\B_Y)$ of
    $\mathcal{L}(Y)$ onto it, which is \eqref{eq:coord-transport}.

    \ref{it:laz-of}
    We prove by induction on the $Y$-length of $c\in\B_Y$ that $c$ and $\Psi(c)$
    have the same status.
    $\Psi((\B_Y)_\bad)=\B_\bad$ then follows from \ref{it:laz-hall} and
    $p\in\B_\good$, and $(\B_Y)_\good < (\B_Y)_\bad$ follows from the fact that the order is induced from $\B$ and $\B_\good < \B_\bad$.

    If $c=y[n]$, the $X$-factorization of $\Psi(c)=\ad_p^n(y)$ has $\Fac(\Psi(c))\subset\{p\}\subset\B_\good$, seed $y$ and single multiplicity $n$.
    By \eqref{eq:of}, $\Psi(c)\in\B_\good$ iff $y\in X_\good$ or $n$ is odd, i.e.\ iff $c\in Y_\good$.

    Let now $c=\ad_{c_r}^{m_r}\cdots\ad_{c_1}^{m_1}(y[n])$ be the $Y$-factorization of $c$. 
    Then
    \begin{equation}
        \label{eq:facto-transport}
        \Psi(c)=\ad_{\Psi(c_r)}^{m_r}\cdots\ad_{\Psi(c_1)}^{m_1}\ad_p^{n}(y),
    \end{equation}
    and, since $\Psi(c_i)\in\B\setminus\{p\}$, the right-hand side is of the form
    \eqref{eq:b-factorization}.
    By uniqueness it is the $X$-factorization of $\Psi(c)$, with factors $\Psi(c_1),\dots,\Psi(c_r)$, together with $p$ if $n>0$, and seed $y$. 
    As $p$ is good, the induction hypothesis makes the first condition in \eqref{eq:of} read ``all $c_i$ are good'' for both alphabets. 
    For the second condition:
    \begin{equation}
        \begin{array}{lll}
            y\in X_\good: & \text{over }X\text{: satisfied (good seed)}; & \text{over }Y\text{: satisfied (}y[n]\in Y_\good\text{)};\\
            y\in X_\bad,\ n\ \text{odd}: & \text{over }X\text{: satisfied (}n\ \text{odd)}; & \text{over }Y\text{: satisfied (}y[n]\in Y_\good\text{)};\\
            y\in X_\bad,\ n\ \text{even}: & \text{over }X\text{: some }m_i\ \text{odd}; & \text{over }Y\text{: some }m_i\ \text{odd}.
        \end{array}
    \end{equation}

    \ref{it:laz-fac} 
    We proceed by induction on the $Y$-length of $c$. 
    For $c=y[n]$, $\Fac^*_Y(c)=\{c\}$ while $\Fac^*(\Psi(c))=\{\Psi(c)\}$ if $n = 0$ and  $\Fac^*(\Psi(c))=\{\Psi(c), p\}$ if $n > 0$. 
    For $c$ as in \eqref{eq:facto-transport},
    \eqref{eq:nested} gives
    \begin{equation}
        \Fac^*_Y(c)=\{c\}\cup\bigcup_i\Fac^*_Y(c_i),
        \qquad
        \Fac^*(\Psi(c)) = \{\Psi(c)\}\cup\bigcup_i\Fac^*(\Psi(c_i))\cup
        \begin{cases}
            \{ p \} & \text{if } n > 0, \\
            \varnothing & \text{otherwise}
        \end{cases}
    \end{equation}
    and one concludes by the induction hypothesis, since $p\notin\Psi(\Br(Y))$.
\end{proof}

\subsection{Proof of the observability result}\label{sec:proof-main}

\begin{proof}[Proof of \cref{prop:observability}]
    By \eqref{eq:nested} and induction on $\abs{b}$, the set $\Fac^*(b)$ is finite. We argue by
    induction on $n:=\abs{\Fac^*(b)}$, the statement of \cref{prop:observability} being quantified over
    all data $(X,\B,K,b)$. If $n=1$, then $\Fac^*(b)=\{b\}$ and \eqref{eq:fac*b-notin-supp-K} is
    \eqref{eq:b-notin-supp-K}. Assume $n>1$.

    \step{Restriction to the letters occurring in $b$}
    Put $S:=\supp\deg_X(b)$ and $\widetilde K:=\rho_S(K)$, a Lie subalgebra of $\mathcal{L}(S)$. 
    By \cref{lem:restriction}, $(S,\B_S,\widetilde K,b)$ is again a datum as in \cref{prop:observability}: an element of $(\B_S)_\bad\setminus\{b\}$ lies in $\B_\bad\setminus\{b\}\subset K$ and in $\mathcal{L}(S)$, hence is fixed by $\rho_S$ and lies in $\widetilde K$; and $\pair{\rho_S(v)}{b}_{\B_S}=\pair{v}{b}_{\B}=0$ for $v\in K$. 
    Since $\Fac^*_S(b)=\Fac^*(b)\subset\B_S$ by \cref{lem:min-fac} \ref{it:min-fac-1}, \eqref{eq:retraction} transfers the conclusion for
    $(S,\widetilde K)$ back to $(X,K)$. 
    We may therefore assume
    \begin{equation}
        \label{eq:full-support}
        X=\supp\deg_X(b).
    \end{equation}

    \step{The pivot}
    By \cref{lem:irreducible}, $b\notin X$ and $p := \min\Fac^*(b)\in X_\good$.
    By \cref{lem:min-fac} \ref{it:min-fac-3} and~\eqref{eq:full-support}, $p = \min X$. 
    Let $Y,\Psi,\Phi,I,\B_Y$ be as in the Lazard elimination \cref{lem:lazard} and put
    \begin{equation}
        \hat b:=\Psi^{-1}(b)\in\B_Y
        \quad \text{and} \quad
        K_Y:=\Phi^{-1}(K\cap I),
    \end{equation}
    which is legitimate since $b\neq p$, and which defines a Lie subalgebra of $\mathcal{L}(Y)$.

    \step{The induction hypothesis after elimination}
    The datum $(Y,\B_Y,K_Y,\hat b)$ satisfies the hypotheses of \cref{prop:observability}. 
    Indeed $\hat b\in(\B_Y)_\bad$ and $(\B_Y)_\bad\setminus\{\hat b\}=\Psi^{-1}(\B_\bad\setminus\{b\})$ by \cref{lem:lazard} \ref{it:laz-of}; the elements of $\B_\bad\setminus\{b\}$ lie in $K$ by \eqref{eq:no-other-bad} and in $I$ by \cref{lem:lazard} \ref{it:laz-coord}, hence $(\B_Y)_\bad\setminus\{\hat b\}\subset K_Y$. 
    Moreover, by \eqref{eq:coord-transport}, $\pair{v}{\hat b}_{\B_Y}=\pair{\Phi(v)}{b}_{\B}=0$ for $v\in K_Y$. 
    Finally \cref{lem:lazard}~\ref{it:laz-fac} gives $\abs{\Fac^*_Y(\hat b)}=n-1$. The induction hypothesis therefore yields
    $\Fac^*_Y(\hat b)\cap\supp_{\B_Y}(K_Y)=\varnothing$, that is, by \eqref{eq:coord-transport} and
    \cref{lem:lazard} \ref{it:laz-fac},
    \begin{equation}
        \label{eq:IH}
        \bigl(\Fac^*(b)\setminus\{p\}\bigr)\cap\supp_{\B}(K\cap I)=\varnothing .
    \end{equation}

    \step{Exclusion of the pivot}
    Assume by contradiction that $p\in\supp_{\B}(K)$. Choose $h\in K$ such that $\pair{h}{p}_{\B}=1$. By \cref{lem:lazard}~\ref{it:laz-coord}, write $h = p + r$ with $r \in I$.
    By \cref{lem:irreducible,thm:support}, there exists
    \begin{equation}
        c\in\bigl(\Fac^*(b)\setminus\{p\}\bigr)
        \cap\supp_{\B}\Lie\bigl(\{p\}\cup(\B_\bad\setminus\{b\})\bigr).
    \end{equation}
    Hence there exists an iterated Lie bracket $q = [\dotsb[[c_1,c_2],\dotsc], c_N]$, where $c_i \in \{p\} \cup (\B_\bad \setminus \{b\})$, and such that $\pair{q}{c}_{\B}\neq0$.
    Let $\beta := \deg_X(c) = \deg_X(q)$.
    
    In the same expression, replace every occurrence of $c_i = p$ by $h$, and denote the resulting element by $\widetilde q$. 
    Since $h\in K$ and $\B_\bad\setminus\{b\}\subset K$, one has $\widetilde q\in K$.
    Moreover, since $c\neq p$, at least one $c_i$ belongs to $\B_\bad\setminus\{b\}\subset I$. 
    As $I$ is an ideal, $\widetilde q\in I$. 
    Thus $\widetilde{q} \in K \cap I$.
    
    Finally, expanding $h=p+r$, we have $\widetilde q=q+R$. 
    Every nonzero $\deg_X$-homogeneous term of $R$ is obtained by replacing at least one occurrence of $p$ by a $\deg_X$-homogeneous component of $r$. 
    Since $r\in I$, it has no component of multidegree $e_p$.
    Thus no such term has multidegree $\beta$. 
    Therefore
    \begin{equation}
        \pair{\widetilde q}{c}_{\B}
        =\pair{q}{c}_{\B}\neq0.
    \end{equation}
    Hence $c \in \supp_{\B}(K\cap I)$, contradicting \eqref{eq:IH}. 
    Therefore
    \begin{equation}
        \label{eq:x-excluded}
        p\notin\supp_{\B}(K).
    \end{equation}

    \step{Conclusion}
    By \eqref{eq:x-excluded} and \cref{lem:lazard} \ref{it:laz-coord}, every $v\in K$ has $\pair{v}{p}_{\B}=0$, hence lies in $I$; so $K=K\cap I$. 
    Now \eqref{eq:IH} excludes $\Fac^*(b)\setminus\{p\}$ from $\supp_{\B}(K)$, and \eqref{eq:x-excluded} excludes $p$.
\end{proof}

\section{Proof of the support result}
\label{sec:support}

We prove \cref{thm:support}.
We need the following definition.

\begin{definition}[Free associative algebra]
    Let $\mathcal{A}(X)$ be the free associative $\R$-algebra on $X$, in which we identify $\mathcal{L}(X)$ with the Lie subalgebra generated by $X$ for the commutator bracket.
    We equip $\mathcal{A}(X)$ with the scalar product $\pair{\cdot}{\cdot}$ for which the words in $X$, including the empty word $\one$, form an orthonormal basis.
    Each $\deg_X$-homogeneous component is then spanned by a finite set of words, so the components are finite-dimensional, pairwise orthogonal, and $\pair{\cdot}{\cdot}$ restricts to a scalar product on each of them; in particular $\pair{\cdot}{\cdot}$ is nondegenerate.
\end{definition}

Throughout this section, $\B$ is a factor-parity Hall set on $X = X_\good \sqcup X_\bad$.
We fix $b \in \B_\bad \setminus X$ such that $\Fac(b) \subset \B_\good$.
We consider the following Lie subalgebra of $\mathcal{L}(X)$:
\begin{equation}
    \mathfrak{h} := \Lie\left(\{p\}\sqcup(\B_\bad\setminus\{b\})\right)
    \quad \text{where} \quad
    p := \min \Fac^*(b) \in X_\good.
\end{equation}
Let $U(\mathfrak{h})$ be the universal enveloping algebra of $\mathfrak{h}$, which we identify with its image in $\mathcal{A}(X)$.
In \cref{sec:duality}, we define a dual Hall coordinate $S_b \in \mathcal{A}(X)$.
In \cref{sec:detection,sec:annihilation} we prove the following incompatible statements.

\begin{proposition}[Detection]
    \label{prop:detection}
    There exists $Z\in U(\mathfrak{h})$ with $\pair{Z}{S_b}\neq0$.
\end{proposition}

\begin{proposition}[Annihilation]
    \label{prop:annihilation}
    If $\supp_{\B}(\mathfrak{h})\cap(\Fac^*(b)\setminus\{p\})=\varnothing$, then $\pair{\cdot}{S_b}=0$ on $U(\mathfrak{h})$.
\end{proposition}

\begin{proof}[Proof of \cref{thm:support}]
    Since the conclusions of \cref{prop:detection,prop:annihilation} are incompatible, one has $\supp_{\B}(\mathfrak{h})\cap(\Fac^*(b)\setminus\{p\})\neq\varnothing$, which is exactly \eqref{eq:fac*b+p-cap}, the conclusion of \cref{thm:support}.
\end{proof}

\subsection{Dual Hall coordinates and residuals}\label{sec:duality}

\subsubsection{Scalar product, coproduct and shuffle}

We consider the tensor product $\mathcal{A}(X) \otimes \mathcal{A}(X)$ which we endow with $(U \otimes V)(U' \otimes V') := UU' \otimes VV'$, and with the scalar product for which the $w \otimes w'$ (for $w,w'$ words) form an orthonormal basis, so that $\pair{U \otimes V}{U' \otimes V'} = \pair{U}{U'}\pair{V}{V'}$.
Let $\Delta : \mathcal{A}(X) \to \mathcal{A}(X) \otimes \mathcal{A}(X)$ be the unique algebra homomorphism such that $\Delta y = y \otimes \one + \one \otimes y$ for $y \in X$.
Every Lie element $z \in \mathcal{L}(X)$ is primitive (i.e.\ satisfies $\Delta z = z \otimes \one + \one \otimes z$) by \cite[Theorem 1.4]{Reutenauer1993}.

For $\deg_X$-homogeneous $F, G \in \mathcal{A}(X)$, the linear form $U \mapsto \pair{\Delta U}{F \otimes G}$ vanishes on every $\deg_X$-homogeneous component except that of multidegree $\deg_X F + \deg_X G$.
By nondegeneracy, there is thus a unique $F \shuff G \in \mathcal{A}(X)$, the \emph{shuffle} of $F$ and $G$, such that
\begin{equation}
    \label{eq:shuffle}
    \pair{U}{F \shuff G} = \pair{\Delta U}{F \otimes G}
    \qquad \text{for all } U \in \mathcal{A}(X).
\end{equation}
The product $\shuff$ is commutative and associative with unit $\one$ and $\deg_X(F \shuff G) = \deg_X F + \deg_X G$ for $F,G$ $\deg_X$-homogeneous.
We then extend it by bilinearity.

\subsubsection{PBW basis and its dual}

A \emph{PBW-monomial} is a product $P=\prod_{c\in\B}^{\searrow}c^{e_c(P)}$ taken in decreasing Hall order, where $e(P)\in\N^{(\B)}$ has finite support.
For PBW-monomials $P,Q$, we write $P\odot Q$ for the PBW-monomial with $e(P\odot Q)=e(P)+e(Q)$.
Since $\mathcal{A}(X)$ is the universal enveloping algebra of $\mathcal{L}(X)$ \cite[Theorem~0.5]{Reutenauer1993}, the Poincaré--Birkhoff--Witt theorem asserts that the PBW-monomials form a basis of $\mathcal{A}(X)$.
They are $\deg_X$-homogeneous, with $\deg_X(P)=\sum_c e_c(P) \deg_X(c)$, so that, for each multidegree $\alpha \in \N^{(X)}$, the finitely many PBW-monomials of multidegree $\alpha$ form a basis of this component of~$\mathcal{A}(X)$.

Consequently, for every PBW-monomial $P$ there is a unique $S_P\in\mathcal{A}(X)$ such that $\pair{Q}{S_P}=\delta_{P,Q}$ for every PBW-monomial $Q$.
It is $\deg_X$-homogeneous of multidegree $\deg_X(P)$, $S_\one=\one$, and $(S_P)_P$ is again a basis of $\mathcal{A}(X)$.
We abbreviate $S_c$ for $c\in\B$.
Thus $\pair{U}{S_P}$ is the coefficient of $P$ in the PBW expansion of $U\in\mathcal{A}(X)$.
In particular, a Lie element expands on the one-term PBW-monomials only, whence
\begin{equation}
    \label{eq:hall-vs-dual}
    \pair{z}{c}_{\B}=\pair{z}{S_c}
    \qquad\text{for all } z\in\mathcal{L}(X),\ c\in\B .
\end{equation}
In other words, $(S_P)_P$ is the bi-orthogonal family to the PBW-monomials for the scalar product $\pair{\cdot}{\cdot}$ on $\mathcal{A}(X)$.
As an example, in a Hall set $\B$ on $X = \{ X_0, X_1 \}$ for which $X_1 < X_0$, one has $(X_1, X_0) \in \B$ and
\begin{equation}
    S_{(X_1, X_0)} = X_1 X_0.
\end{equation}
Indeed, for the one-term PBW-monomial $(X_1, X_0)$ and the two-terms PBW monomial $X_0 X_1$,
\begin{equation}
    \pair{[X_1, X_0]}{S_{(X_1, X_0)}}
    = \pair{X_1 X_0 - X_0 X_1}{X_1 X_0} = 1
    \quad \text{and} \quad 
    \pair{X_0 X_1}{S_{(X_1, X_0)}}
    = 0.
\end{equation}

We now recall two properties of the basis dual to the PBW basis associated with a Hall set; see \cite[Theorem 5.3]{Reutenauer1993}.
Since \cite{Reutenauer1993} uses a different Hall set convention, we provide a proof in \cref{app:dual}.

\begin{lemma}
    \label{lem:dual-product}
    For PBW-monomials $P,Q$,
    \begin{equation}
        \label{eq:dual-product}
        S_P\shuff S_Q=\Bigl(\prod_{c\in\B}\binom{e_c(P)+e_c(Q)}{e_c(P)}\Bigr)S_{P \odot Q},
        \qquad\text{hence}\qquad
        S_P=\mathop{\shuff}_{c\in\B}\frac{S_c^{\shuff e_c(P)}}{e_c(P)!}.
    \end{equation}
\end{lemma}

\begin{lemma}
    \label{lem:dual-hall}
    Let $c\in\B$ and write its factorization as
    $c=\ad_{c_r}^{m_r}\dotsb\ad_{c_1}^{m_1}(\seed(c))$.
    By \cref{lem:facto-Hall}, $c_1<\dotsb<c_r$, so that
    $M_c:=c_r^{m_r}\cdots c_1^{m_1}$ is a PBW-monomial.
    Then
    \begin{equation}
        \label{eq:dual-hall}
        S_c = G_c\seed(c),
        \quad\text{where}\quad
        G_c := S_{M_c}
        \overset{\eqref{eq:dual-product}}{=}
        \frac{S_{c_1}^{\shuff m_1}}{m_1!}\shuff\cdots\shuff
        \frac{S_{c_r}^{\shuff m_r}}{m_r!},
    \end{equation}
    the multiplication by $\seed(c)$ being the concatenation by the letter $\seed(c)$.
    In particular, every word occurring in $S_c$ ends with $\seed(c)$.
    For $c\in X$, one has $r=0$, $M_c=\one$, $G_c=\one$ and $S_c=c$.
\end{lemma}

\subsubsection{Residuals}

For $y\in X$, the residuals $\ell_y,r_y:\A(X)\to\A(X)$ are the linear maps defined on words by $\ell_y(yw)=w$, $r_y(wy)=w$, and by $0$ on words not having the indicated first (resp.\ last) letter.
For $a\in\mathcal{L}(X)$, let $\partial_a:\A(X)\to\A(X)$ be the adjoint of the left multiplication by $a$, i.e.\ the unique linear map such that
\begin{equation}
    \label{eq:def-partial}
    \pair{U}{\partial_a F}=\pair{a U}{F}
    \qquad \text{for all} \quad U, F \in \A(X);
\end{equation}
it is well defined because the $\deg_X$-homogeneous components of $\A(X)$ are finite-dimensional and pairwise orthogonal.
Finally, for $a\in\A(X)$, set $D_a(U):=aU-Ua$, so that $D_a=\ad_a$ on $\mathcal{L}(X)$.

\begin{lemma}
    \label{lem:residuals}
    Let $y\in X$, $a\in\mathcal{L}(X)$, $F,G,U,V\in\A(X)$ and $n\in\N$.
    Then
    \begin{align}
        &\pair{yU}{F}=\pair{U}{\ell_yF},
        \quad
        \pair{Uy}{F}=\pair{U}{r_yF};
        \label{eq:transpose}\\
        &\partial_y=\ell_y
        \quad \text{and} \quad
        \pair{D_y^n(U)}{F}=\pair{U}{(\ell_y-r_y)^nF};
        \label{eq:Dp}\\
        &r_y(UV)=U\,r_y(V)
        \quad\text{if $V$ has zero constant term};
        \label{eq:right-factor}\\
        &\partial_a(F\shuff G)=(\partial_a F)\shuff G+F\shuff(\partial_a G);
        \label{eq:derivation}\\
        &\partial_a (Fy)=(\partial_a F)y
        \quad\text{if $a$ and $F$ are $\deg$-homogeneous with }\deg F\geq\deg a ,
        \label{eq:past-seed}
    \end{align}
    where $\deg$ denotes the total degree.
\end{lemma}

\begin{proof}
    The two identities \eqref{eq:transpose} are read off the definitions on words.
    \eqref{eq:right-factor} is likewise immediate on words.
    By \eqref{eq:def-partial} and \eqref{eq:transpose}, $\partial_y=\ell_y$ and $\ell_y-r_y$ is the adjoint of $D_y$, whence \eqref{eq:Dp}, the adjoint of $D_y^n$  being the $n$-th power of the adjoint of $D_y$.

    For \eqref{eq:derivation}, let $U\in\A(X)$.
    Since $\Delta$ is a homomorphism and $a$ is primitive, \eqref{eq:def-partial} and \eqref{eq:shuffle} give
    \begin{equation}
        \begin{split}
        \pair{U}{\partial_a(F\shuff G)}
        =\pair{\Delta(a U)}{F\otimes G}
        &=\pair{(a\otimes\one+\one\otimes a)\Delta U}{F\otimes G}
        \\&=\pair{\Delta U}{(\partial_aF)\otimes G+F\otimes(\partial_aG)},
        \end{split}
    \end{equation}
    using \eqref{eq:def-partial} on each tensor factor.
    By \eqref{eq:shuffle} again, this is $\pair{U}{(\partial_a F)\shuff G+F\shuff(\partial_a G)}$.

    For \eqref{eq:past-seed}, both sides are $\deg$-homogeneous of degree $\deg F+1-\deg a\geq1$, hence determined by their scalar products against $\deg$-homogeneous elements $U$ of positive degree. 
    For such $U$, \eqref{eq:transpose}, \eqref{eq:right-factor} and~\eqref{eq:def-partial} give $\pair{U}{\partial_a(Fy)}=\pair{r_y(a U)}{F}=\pair{a r_y(U)}{F} = \pair{r_y(U)}{\partial_a F}=\pair{U}{(\partial_a F)y}$.
\end{proof}

\subsection{Annihilation}\label{sec:annihilation}

\begin{lemma}[Factor residual]
    \label{lem:factor-residual}
    Let $a, c \in \B$ with $a \notin \Fac^*(c)$.
    Then $\partial_a S_c = 0$.
    
    Consequently $\partial_u S_c = 0$ for every $u \in \mathcal{L}(X)$ such that $\supp_\B(u) \cap \Fac^*(c) = \varnothing$.
\end{lemma}

\begin{proof}
    Note that $a \neq c$, since $c \in \Fac^*(c)$.
    We fix $a$ and proceed by induction on $\abs{c}$.

    If $\abs{a} \geq \abs{c}$, then $\partial_a S_c$ is $\deg$-homogeneous of degree $\abs{c}-\abs{a} \leq 0$,
    hence vanishes unless $\abs{a} = \abs{c}$, in which case
    $\partial_a S_c = \pair{a}{S_c}\one = \delta_{a,c}\one = 0$,
    since $a$ is a one-term PBW-monomial.

    If $\abs{a} < \abs{c}$, then $c \notin X$.
    Write $c = \ad_{c_r}^{m_r} \dotsb \ad_{c_1}^{m_1}(\seed(c))$ with $r \geq 1$.
    By \eqref{eq:nested}, $\Fac^*(c_j) \subset \Fac^*(c)$, so $a \notin \Fac^*(c_j)$ and
    $\partial_a S_{c_j} = 0$ by the induction assumption.
    Since $G_c$ is a shuffle product of the $S_{c_j}$ by \eqref{eq:dual-hall} and $\partial_a$ is a
    shuffle derivation by \eqref{eq:derivation}, $\partial_a G_c = 0$.
    As $\deg G_c = \abs{c} - 1 \geq \abs{a}$, \eqref{eq:past-seed} applies and yields
    $\partial_a S_c = \partial_a(G_c \seed(c)) = (\partial_a G_c)\seed(c) = 0$.

    The last assertion follows by linearity, expanding $u$ in the basis $\eval(\B)$.
\end{proof}

\begin{proof}[Proof of \cref{prop:annihilation}]
    Let $\mathfrak{h}_0 := \{ u \in \mathfrak{h} \mid \pair{u}{p}_\B = 0 \}$.
    Since $p \in \mathfrak{h}$ and $\pair{p}{p}_\B = 1$, the linear form $\pair{\cdot}{p}_\B$ is nonzero on $\mathfrak{h}$ and $\mathfrak{h} = \R p \oplus \mathfrak{h}_0$.
    For $u \in \mathfrak{h}_0$, one has $p \notin \supp_\B(u)$ by definition of $\mathfrak{h}_0$, while $\supp_\B(u) \cap (\Fac^*(b) \setminus \{p\}) = \varnothing$ by assumption.
    Hence
    \begin{equation}
        \label{eq:h0-support}
        \supp_{\B}(\mathfrak{h}_0) \cap \Fac^*(b) = \varnothing .
    \end{equation}
    Choose a totally ordered basis of $\mathfrak{h}$ consisting of $p$ and of a basis of $\mathfrak{h}_0$, with $p$ as its
    smallest element.
    Since PBW-monomials are written in decreasing order, the PBW theorem for
    $U(\mathfrak{h}) \subset \mathcal{A}(X)$ shows that $U(\mathfrak{h})$ is spanned by the products
    $u_1 \cdots u_k\, p^m$ with $m, k \geq 0$ and $u_1, \dotsc, u_k \in \mathfrak{h}_0$.
    It therefore suffices to prove that $\pair{u_1 \cdots u_k p^m}{S_b} = 0$ for such products.

    If $k \geq 1$, then $\partial_{u_1} S_b = 0$ by \eqref{eq:h0-support} and
    \cref{lem:factor-residual}, so
    $\pair{u_1 \cdots u_k p^m}{S_b} = \pair{u_2 \cdots u_k p^m}{\partial_{u_1} S_b} = 0$.
    If $k = 0$, then $p^m$ is a PBW-monomial with $m$ factors whereas $b$ has exactly one, and $p \neq b$ because $b \notin X$.
    Hence $p^m \neq b$ and $\pair{p^m}{S_b} = 0$.
\end{proof}

\subsection{Detection}\label{sec:detection}

We prove \cref{prop:detection}.
Let us fix some notation.
Write the factorization
\begin{equation}
    b = \ad_{a_r}^{m_r} \dotsb \ad_{a_1}^{m_1}(z)
    \quad \text{where} \quad 
    z := \seed(b) \in X.
\end{equation}
Since $b\in\B_\bad\setminus X$ and $\Fac(b)\subset\B_\good$, \eqref{eq:of} gives
\begin{equation}
    \label{eq:even}
    r\geq1,\qquad z\in X_\bad,\qquad m_1,\dotsc,m_r\in2\N.
\end{equation}
Moreover, $\nu(b) := \deg_p(b) \ge 1$, since $p\in\Fac^*(b)$ occurs in $b$ by \cref{lem:min-fac}~\ref{it:min-fac-1}.

We introduce
\begin{equation}
    \label{eq:Mb-F}
    M_b:=a_r^{m_r}\cdots a_1^{m_1}
    \quad \text{and} \quad
    F := \ell_p^{\nu(b)} S_{M_b}.
\end{equation}
The proof of \cref{prop:detection} uses crucially the following total order on the PBW-monomials.

\begin{definition}[PBW order]
    \label{def:pbw-order}
    Let $P\neq Q$ be PBW-monomials. 
    The set $\{c\in\B\mid e_c(P)\neq e_c(Q)\}$ is nonempty and finite.
    Let $c$ be its largest element.
    We set $P \prec Q \iff e_c(P)<e_c(Q)$.
    
    In particular, if $P \prec Q$, then $P \odot R \prec Q \odot R$ for any PBW-monomial $R$.
\end{definition}

We will prove \cref{prop:detection} from the following results.

\begin{lemma}[Target]
    \label{lem:target}
    One has $F\neq0$. 
    Let $M$ be the $\prec$-largest PBW-monomial such that $\pair{M}{F}\neq0$. 
    Then $\deg{M} = \abs{b} - 1 - \nu(b)$ and $e_c(M)$ is even for all $c \in \B$.
\end{lemma}

\begin{lemma}[Detector]
    \label{lem:detector}
    Let $M$ be as in \cref{lem:target}.
    There exists $W\in U(\mathfrak{h})$ such that
    \begin{equation}
        \label{eq:rzW=M}
        r_z(W) = M + \sum_{M \prec Q} \alpha_Q Q
    \end{equation}
    where the sum is finite, indexed by PBW-monomials $Q$, and $\alpha_Q \in \R$.
\end{lemma}

\begin{proof}[Proof of \cref{prop:detection}]
    Let $W$ be given by \cref{lem:detector}.
    Set $n := \nu(b)$ and $Z := D_p^n W$.
    Since $W \in U(\mathfrak{h})$ and $p \in \mathfrak{h}$, one has $Z \in U(\mathfrak{h})$.
    Moreover, using \eqref{eq:Dp},
    \begin{equation}
        \pair{Z}{S_b}
        = \pair{D_p^n W}{S_b}
        = \pair{W}{(\ell_p - r_p)^n S_b}.
    \end{equation}
    By \cref{lem:dual-hall}, $S_b = S_{M_b} z$, and $z\neq p$ because $z\in X_\bad$ and $p\in X_\good$.
    Hence, by \eqref{eq:Mb-F},
    \begin{equation}
        (\ell_p - r_p)^n S_b = \ell_p^n S_b = (\ell_p^n S_{M_b}) z = F z
    \end{equation}
    Thus $\pair{Z}{S_b} = \pair{W}{F z} = \pair{r_z(W)}{F}$ by \eqref{eq:transpose}.
    Since $M$ is by definition the $\prec$-largest PBW-monomial such that $\pair{M}{F} \neq 0$, \eqref{eq:rzW=M} entails that $\pair{Z}{S_b} = \pair{r_z(W)}{F} = \pair{M}{F} \neq 0$.
\end{proof}

\subsubsection{The target}\label{sec:target}

We prove \cref{lem:target}.
We start with the following intermediate lemma.

\begin{lemma}[Leading term of a shuffle]
    \label{lem:max-shuffle}
    Let $F_1,\dots,F_N\in\A(X)\setminus\{0\}$, with finite dual-PBW expansions $F_j=\sum_P\alpha_{j,P}S_P$, and set $P_j:=\max_\prec\{P\mid\alpha_{j,P}\neq0\}$ and $P_\star:=P_1\odot\cdots\odot P_N$. Then
    \begin{equation}
        \label{eq:max-shuffle}
        F_1\shuff\cdots\shuff F_N=\alpha S_{P_\star}+\sum_{Q\prec P_\star}\alpha_QS_Q,
        \quad \text{where} \quad
        \alpha=\Bigl(\prod_{j=1}^N\alpha_{j,P_j}\Bigr)
        \prod_{c\in\B}\frac{\bigl(\sum_je_c(P_j)\bigr)!}{\prod_je_c(P_j)!}\neq0 .
    \end{equation}
    In particular, a shuffle product of nonzero elements is nonzero.
\end{lemma}

\begin{proof}
    Expand each $F_j$ in the dual-PBW basis. 
    By \eqref{eq:dual-product}, every tuple $(Q_1,\dots,Q_N)$ with $\alpha_{j,Q_j}\neq0$ contributes a nonzero scalar multiple of $S_{Q_1\odot\cdots\odot Q_N}$, and $Q_j\preceq P_j$ by maximality so $Q_1\odot\cdots\odot Q_N\preceq P_\star$, with strict inequality as soon as one $Q_j\prec P_j$. 
    Hence $S_{P_\star}$ is produced by the single tuple $(P_1,\dots,P_N)$, and iterating \eqref{eq:dual-product} gives $\alpha$.
\end{proof}

For $c\in\B$, we define $T_c := \ell_p^{\nu(c)} S_c$, where $\nu(c) := \deg_p(c)$.

\begin{lemma}[Saturation formula]
    \label{lem:saturation}
    Let $c\in\B\setminus X$, with factorization $c=\ad_{c_s}^{q_s}\dotsb\ad_{c_1}^{q_1}(\seed(c))$, and set
    $M_c:=c_s^{q_s}\cdots c_1^{q_1}$.
    Assume $\seed(c)\neq p$.
    Then $\nu(c)=\sum_{j=1}^s q_j\nu(c_j)\leq\abs{c}-1$ and
    \begin{equation}
        \label{eq:saturation}
        T_c=\bigl(\ell_p^{\nu(c)}S_{M_c}\bigr)\seed(c),
        \qquad
        \ell_p^{\nu(c)}S_{M_c}=\gamma_c\mathop{\shuff}_{j=1}^s T_{c_j}^{\shuff q_j},
        \qquad
        \gamma_c:=\frac{\nu(c)!}{\prod_j q_j!\,(\nu(c_j)!)^{q_j}}>0 .
    \end{equation}
\end{lemma}

\begin{proof}
    Since $\seed(c)\neq p$, counting occurrences of $p$ in the factorization of $c$ gives $\nu(c)=\sum_jq_j\nu(c_j)\leq\abs{c}-1$.
    By \cref{lem:dual-hall}, $S_c=S_{M_c}\seed(c)$ and $\deg S_{M_c}=\abs{c}-1$.
    Hence \eqref{eq:past-seed} may be applied successively $\nu(c)$ times, yielding the first identity in
    \eqref{eq:saturation}.

    Moreover, by \eqref{eq:dual-product},
    \begin{equation}
        S_{M_c}
        =\frac1{\prod_jq_j!}
        \mathop{\shuff}_{j=1}^s S_{c_j}^{\shuff q_j}.
    \end{equation}
    Since $\ell_p=\partial_p$ is a shuffle derivation, its $\nu(c)$-th power expands by the multinomial Leibniz rule.
    By $\deg_X$-homogeneity, $\ell_p^mS_{c_j}=0$ for $m>\nu(c_j)$; since $\nu(c)=\sum_jq_j\nu(c_j)$, the only surviving terms are those in which each copy of $S_{c_j}$ receives exactly $\nu(c_j)$ residuals. 
    Their multinomial coefficient is $\nu(c)! / \prod_j (\nu(c_j)!)^{q_j}$, which gives the second identity.
\end{proof}

\begin{lemma}[Non-vanishing]
    \label{lem:saturation-nonzero}
    For every $c\in\Fac^*(b)$, one has $T_c\neq0$, and $T_p=\one$.
\end{lemma}

\begin{proof}
    We argue by induction on $\abs{c}$. 
    If $c\in X$, then $T_p=\ell_p p=\one$ and $T_c=c$ otherwise.
    Let $c\in\Fac^*(b)\setminus X$. Its factors $c_j$ lie in $\Fac^*(b)$ by \eqref{eq:nested}, and
    $p=\min\Fac^*(b)\leq c_1<\seed(c)$ by \cref{lem:facto-Hall}, so $\seed(c)\neq p$ and \cref{lem:saturation}
    applies.
    The $T_{c_j}$ are nonzero by the induction hypothesis, hence so is their shuffle by \cref{lem:max-shuffle}.
    The first identity in \eqref{eq:saturation}, together with injectivity of right concatenation by $\seed(c)$, therefore gives $T_c\neq0$.
\end{proof}

\begin{proof}[Proof of \cref{lem:target}]
    By \cref{lem:saturation} applied to $c=b$ and \eqref{eq:Mb-F}, one has $F = \gamma_b \mathop{\shuff}_j T_{a_j}^{\shuff m_j}$ with $\gamma_b>0$.
    By \cref{lem:saturation-nonzero}, $T_{a_j}\neq0$ for every $j$. 
    Let $P_j$ be the $\prec$-largest PBW-monomial with $\pair{P_j}{T_{a_j}}\neq0$.
    By \cref{lem:max-shuffle}, $F\neq0$ and its largest dual-PBW term is indexed by
    \begin{equation}
        \label{eq:M-shuffle}
        M=P_1^{\odot m_1}\odot\cdots\odot P_r^{\odot m_r},
        \qquad\text{so that}\qquad
        e_c(M)=\sum_{j=1}^rm_je_c(P_j),
    \end{equation}
    which is even for every $c\in\B$ because every $m_j$ is even by \eqref{eq:even}. 
    Finally $\deg P_j=\deg T_{a_j}=\abs{a_j}-\nu(a_j)$, while $z\neq p$ gives $\abs{b}=1+\sum_j m_j\abs{a_j}$ and $\nu(b) = \sum_j m_j\nu(a_j)$, so $\deg M=\abs{b}-1-\nu(b)$.
\end{proof}

\subsubsection{The detector}\label{sec:detector}

Since $\B_\good < \B_\bad$, the PBW-monomial $M$ factors uniquely as $M=M_\bad M_\good$, every term in $M_\bad$ being bad, with
\begin{equation}
    \label{eq:Mgood}
    M_\good=g_t^{2k_t}\cdots g_1^{2k_1},
    \qquad
    g_1<\cdots<g_t\ \text{in }\B_\good,
    \qquad k_j\geq1.
\end{equation}
We define
\begin{equation}
    \label{eq:def-W}
    d := \ad_{g_t}^{2k_t}\cdots\ad_{g_1}^{2k_1}(z)
    \quad \text{and} \quad
    W := M_\bad d.
\end{equation}
When $t = 0$, $M_\good = \one$ and $d = z$.

\begin{proof}[Proof of \cref{lem:detector}]
    \step{The element $W$ lies in $U(\mathfrak{h})$}
    First, $d \in \B_\bad$.
    This is immediate when $t = 0$.
    Assume $t>0$, let $u_1\leq\cdots\leq u_K$ be the sequence containing $2k_j$ copies of $g_j$ for each $j$, and set $d^{(0)}:=z$ and $d^{(i)}:=(u_i,d^{(i-1)})$, so that $d=d^{(K)}$. Every $d^{(i)}$ lies in $\B$: for $i=1$ this follows from $u_1\in\B_\good<z\in\B_\bad$ and $z\in X$; for $i\geq2$ one has $\lambda(d^{(i-1)})=u_{i-1}\leq u_i$, and $u_i<d^{(i-1)}$ either by the third Hall axiom when $u_i=u_{i-1}$, or, when $u_i>u_{i-1}$, because $d^{(i-1)}$ then has good factors, even multiplicities and seed $z\in X_\bad$, hence lies in $\B_\bad$ by \eqref{eq:of} while $\B_\good<\B_\bad$; in both cases the second Hall axiom concludes. 
    The same application of \eqref{eq:of} at $i=K$ gives $d\in\B_\bad$.

    By \cref{lem:target}, $\deg W=\deg M+1=\abs{b}-\nu(b)<\abs{b}$. 
    Every factor of $M_\bad$, as well as $d$, is therefore in $\B_\bad$ and of degree strictly smaller than $\abs{b}$, hence lies in $\B_\bad\setminus\{b\}\subset\mathfrak{h}$, and $W\in U(\mathfrak{h})$.

    \step{Smallest PBW term of $r_z(W)$} 
    By \eqref{eq:right-factor}, $r_z(W) = M_\bad r_z(d)$ and, for $\deg$-homogeneous $u,v$ of positive degree, $r_z([u,v]) = u r_z(v) - v r_z(u)$.
    If $t=0$, then $r_z(d)=r_z(z)=\one$, so $r_z(W)=M$. 
    Assume $t>0$. 
    For $0\leq i\leq K$, set
    \begin{equation}
        A_i:=M_\bad u_K\cdots u_{i+1}r_z(d^{(i)}).
    \end{equation}
    Then $A_0 = M$ since $r_z(z)=\one$ and $A_K = r_z(W)$.
    Moreover, for $1\leq i\leq K$,
    \begin{equation}
        A_i-A_{i-1} = -M_\bad u_K\cdots u_{i+1}d^{(i-1)}r_z(u_i).
    \end{equation}
    Expand $r_z(u_i)$ in the PBW basis. For every PBW-monomial $R$ occurring in this expansion, the product $M_\bad u_K\cdots u_{i+1}d^{(i-1)}R$ agrees with $M=M_\bad u_K\cdots u_1$ up to the position occupied by $u_i$, where $u_i$ is replaced by the strictly larger element $d^{(i-1)}$. 
    By \cref{lem:app-pbw-triangular}, every PBW-monomial occurring in its PBW expansion is therefore strictly larger than $M$. 
    Thus every PBW-monomial occurring in $A_i-A_{i-1}$ is $\succ M$. 
    Summing over $i$ yields \eqref{eq:rzW=M}.
\end{proof}

\appendix

\section{Comparison with classical sufficient conditions}
\label{sec:compare-sufficient}

We first show that, in our sufficient condition, only the bad brackets having no bad factor need to be compensated.
We then compare this irreducible family with the bad brackets occurring in the conditions of Sussmann, Agrachev--Gamkrelidze and Krastanov.

Throughout the appendix, $X_\good = \{ X_1, \dotsc, X_q \}$, $X_\bad = \{ X_0 \}$ and $\B = \B_\good \sqcup \B_\bad$ is a factor-parity Hall set on $X$.
We recall the family of Sussmann-bad brackets
\begin{equation}
    \label{eq:Sbad-app}
    S_\bad
    :=
    \left\{ b \in \Br(X) \mid n_0(b) \text{ is odd and } n_i(b) \text{ is even for } 1 \leq i \leq q \right\}.
\end{equation}

\subsection{Parity of bad brackets}

All the parity properties we need follow from the following lemma, which we state over an arbitrary alphabet since we will also apply it after Lazard elimination.

\begin{lemma}[Parity]
    \label{lem:bad-parity}
    Let $\mathcal{H} = \mathcal{H}_\good \sqcup \mathcal{H}_\bad$ be a factor-parity Hall set on an alphabet $Y = Y_\good \sqcup Y_\bad$ and let $b \in \mathcal{H}_\bad$, with factorization $b = \ad_{a_r}^{m_r} \dotsb \ad_{a_1}^{m_1}(\seed(b))$.
    Then
    \begin{enumerate}[label=\textup{(\roman*)},leftmargin=2em]
        \item \label{it:parity-seed} $\seed(b) \in Y_\bad$, and $m_j$ is even for every good factor $a_j$;
        \item \label{it:parity-letters} $n_y(b)$ is even for every $y \in Y_\good$;
        \item \label{it:parity-irr} if moreover $\Fac(b) \subset \mathcal{H}_\good$, then every $m_j$ is even and $\sum_{y \in Y_\bad} n_y(b)$ is odd.
    \end{enumerate}
\end{lemma}

\begin{proof}
    \ref{it:parity-seed}
    Let $k$ be minimal such that $a_k \in \mathcal{H}_\bad$, with $k := r+1$ if all factors are good, and set
    \begin{equation}
        c := \ad_{a_{k-1}}^{m_{k-1}} \dotsb \ad_{a_1}^{m_1}(\seed(b)) \in \mathcal{H}.
    \end{equation}
    Then $c \in \mathcal{H}_\bad$: either $c = b$, or $a_k < c$ by the second Hall axiom and $\mathcal{H}_\good < \mathcal{H}_\bad$.
    All the factors of $c$ are good, so \eqref{eq:of} applied to $c$ yields $\seed(b) = \seed(c) \in Y_\bad$ and $m_1, \dotsc, m_{k-1}$ even.
    Since $\mathcal{H}_\good < \mathcal{H}_\bad$, the brackets $a_1, \dotsc, a_{k-1}$ are exactly the good factors of $b$.

    \ref{it:parity-letters} Induction on $\abs{b}$, using $n_y(b) = n_y(\seed(b)) + \sum_j m_j n_y(a_j)$: by \ref{it:parity-seed} the seed contributes nothing, the good factors occur with even multiplicity, and each bad factor contains an even number of occurrences of $y$ by the induction hypothesis.

    \ref{it:parity-irr} Here $k = r+1$, so all the $m_j$ are even by \ref{it:parity-seed}.
    Writing $\sigma(a) := \sum_{y \in Y_\bad} n_y(a)$, one gets $\sigma(b) = 1 + \sum_j m_j \sigma(a_j)$, which is odd.
\end{proof}

Specialized to $X_\bad = \{ X_0 \}$, items \ref{it:parity-letters} and \ref{it:parity-irr} read as follows.

\begin{corollary}
    \label{cor:bad-in-Sbad}
    If $b \in \B_\bad$ satisfies $\Fac(b) \subset \B_\good$, then $b \in S_\bad$.
\end{corollary}

\subsection{Irreducible bad brackets}

We define the set of \emph{irreducible} bad brackets as
\begin{equation}
    \label{eq:def-Pbad}
    \B_\bad^\irr := \B_\bad \setminus (\B_\bad, \B_\bad).
\end{equation}

\begin{lemma}
    \label{lem:bad-irreducible-decomposition}
    Every $b \in \B_\bad^\irr$ satisfies $\Fac(b) \subset \B_\good$.
\end{lemma}

\begin{proof}
    Write $b = \ad_{a_r}^{m_r} \dotsb \ad_{a_1}^{m_1}(\seed(b))$ and assume that some factor is bad; then so is $a_r$, since $a_1 < \dotsb < a_r$ and $\B_\good < \B_\bad$.
    Set $c := \ad_{a_r}^{m_r-1} \ad_{a_{r-1}}^{m_{r-1}} \dotsb \ad_{a_1}^{m_1}(\seed(b)) \in \B$.
    The second Hall axiom gives $a_r < c$, hence $c \in \B_\bad$ and $b = (a_r,c) \in (\B_\bad,\B_\bad)$, a contradiction.
\end{proof}

We now show that the \emph{reducible} bad brackets of $(\B_\bad,\B_\bad)$ do not have to be compensated separately. 
Fix now $\theta \in (0,1)$, recall the weight $\omega_\theta(b) = \theta n_0(b) + n_1(b) + \dotsb + n_q(b)$ and set, for $r > 0$,
\begin{align}
    \mathcal{L}_{<r}
    &:=
    \vect\{a\in\B\mid\omega_\theta(a)<r\},
    \\
    \mathcal{L}_{\leq r}
    &:=
    \vect\{a\in\B\mid\omega_\theta(a)\leq r\},
    \\
    \mathcal{G}_{<r}
    &:=
    \vect\{a\in\B_\good\mid\omega_\theta(a)<r\}.
\end{align}
Since $\omega_\theta$ is additive with respect to the multidegree, the Hall expansion of a bracket preserves the weight, so that
\begin{equation}
    \label{eq:weight-filtration-bracket}
    [\mathcal{L}_{\leq r}, \mathcal{L}_{<s}] \subset \mathcal{L}_{<r+s}.
\end{equation}
Finally, for a control-affine system \eqref{eq:syst}, let
\begin{equation}
    \label{eq:def-Kf}
    K_f := \{ z \in \mathcal{L}(X) \mid f_z(0) = 0 \},
\end{equation}
which is a Lie subalgebra of $\mathcal{L}(X)$, as already used in the proof of \cref{thm:necessary}.

For $b \in \B_\bad$, our compensation condition \eqref{eq:fb0-compensated} reads
\begin{equation}
    \label{eq:compensation-mod-Kf}
    b \in K_f + \mathcal{G}_{<\omega_\theta(b)},
\end{equation}
while Sussmann's condition of \cref{thm:sussmann} reads
\begin{equation}
    \label{eq:sussmann-mod-Kf}
    \forall b \in S_\bad, \qquad b \in K_f + \mathcal{L}_{<\omega_\theta(b)}.
\end{equation}

The next lemma contains both comparisons: it upgrades a compensation of only the irreducible bad brackets by \emph{arbitrary} brackets of lower weight into a compensation of \emph{all} bad brackets by \emph{good} brackets of lower weight.

\begin{lemma}[Absorption]
    \label{lem:absorption}
    Let $\theta \in (0,1)$ and assume that $b \in K_f + \mathcal{L}_{<\omega_\theta(b)}$ for every $b \in \B_\bad^\irr$.
    Then $b \in K_f + \mathcal{G}_{<\omega_\theta(b)}$ for every $b \in \B_\bad$.
\end{lemma}

\begin{proof}
    Assume by contradiction that some bad bracket fails \eqref{eq:compensation-mod-Kf} and choose one, say $b$, of minimal weight $\omega := \omega_\theta(b)$.
    Such a choice is possible because $\omega_\theta(a) \geq \theta \abs{a}$, so that only finitely many brackets have weight below a given bound.
    We claim that, in any case,
    \begin{equation}
        \label{eq:b-in-Kf-L}
        b \in K_f + \mathcal{L}_{<\omega}.
    \end{equation}
    If $b \in \B_\bad^\irr$, this is the assumption.
    Otherwise $b = (a,c)$ with $a, c \in \B_\bad$, and $\omega_\theta(a), \omega_\theta(c) < \omega$ by positivity and additivity of the weight.
    By minimality of $b$, one can write $a = k_a + g_a$ and $c = k_c + g_c$ with $k_a, k_c \in K_f$, $g_a \in \mathcal{G}_{<\omega_\theta(a)}$ and $g_c \in \mathcal{G}_{<\omega_\theta(c)}$.
    In particular $k_a \in \mathcal{L}_{\leq \omega_\theta(a)}$ and $k_c \in \mathcal{L}_{\leq \omega_\theta(c)}$.
    Hence
    \begin{equation}
        b = [a,c] = [k_a,k_c] + \bigl( [k_a,g_c] + [g_a,k_c] + [g_a,g_c] \bigr),
    \end{equation}
    where $[k_a,k_c] \in K_f$ since $K_f$ is a Lie subalgebra, while the parenthesis lies in $\mathcal{L}_{<\omega}$ by \eqref{eq:weight-filtration-bracket}.
    This proves \eqref{eq:b-in-Kf-L}.

    Expand now in the Hall basis the component of $b$ lying in $\mathcal{L}_{<\omega}$.
    Its good terms belong to $\mathcal{G}_{<\omega}$, and each of its bad terms $d$ satisfies $\omega_\theta(d) < \omega$, hence $d \in K_f + \mathcal{G}_{<\omega_\theta(d)} \subset K_f + \mathcal{G}_{<\omega}$ by minimality of $b$.
    Therefore $b \in K_f + \mathcal{G}_{<\omega}$, a contradiction.
\end{proof}

\begin{proposition}[Compensation of irreducible bad brackets]
    \label{prop:irreducible-compensation}
    Let $\theta \in (0,1)$.
    Then \eqref{eq:compensation-mod-Kf} holds for every $b \in \B_\bad$ if and only if it holds for every $b \in \B_\bad^\irr$.
\end{proposition}

\begin{proof}
    The converse implication follows from \cref{lem:absorption} and $\mathcal{G}_{<r} \subset \mathcal{L}_{<r}$.
\end{proof}

Thus the family of bad brackets requiring \emph{direct} compensation in our sufficient condition is $\B_\bad^\irr$, not the whole of $\B_\bad$: reducible bad brackets are obtained by bracketing two smaller bad brackets, and their compensation follows recursively from the fact that $K_f$ is a Lie subalgebra.

\subsection{Sussmann's bad brackets}
\label{sec:compare-sussmann}

It is not true that $\B_\bad \subset S_\bad$: bracketing two bad brackets need not preserve the parity of $n_0$.
By \cref{lem:bad-irreducible-decomposition,cor:bad-in-Sbad}, however, all the bad brackets requiring direct compensation $\B_\bad^\irr$ do belong to $S_\bad$, which yields the following comparison.

\begin{corollary}
    \label{cor:sussmann-implies-ours}
    If a system \eqref{eq:syst} satisfies \eqref{eq:sussmann-mod-Kf} for some $\theta \in (0,1)$, then it satisfies \eqref{eq:fb0-compensated} for the same $\theta$.
    Hence \cref{thm:relaxed-sufficient} applies to every system to which \cref{thm:sussmann} applies.
\end{corollary}

\begin{proof}
    By \cref{lem:bad-irreducible-decomposition,cor:bad-in-Sbad}, $\B_\bad^\irr \subset S_\bad$, so \eqref{eq:sussmann-mod-Kf} provides the assumption of \cref{lem:absorption}.
\end{proof}

\subsection{Agrachev--Gamkrelidze and Krastanov filters}
\label{sec:compare-AGK}

Both the Agrachev--Gamkrelidze sufficient condition of~\cite{AgrachevGamkrelidze1993_Semigroups} and the Krastanov sufficient condition of~\cite{Krastanov2009} start from a given set $\Pi$ satisfying some properties.
A natural way to construct this set $\Pi$ is to obtain it by Lazard elimination, as done by the authors in their examples.
The purpose of this paragraph is to compare, when $\Pi$ is obtained in that way, the sets of brackets considered as bad by Agrachev--Gamkrelidze $\mathrm{AG}_\bad$, by Krastanov $\mathrm{K}_\bad$ and by a factor-parity Hall set $\B_\bad^\irr$.
We do not claim that our \cref{thm:relaxed-sufficient} implies theirs for two different reasons:
\begin{itemize}
    \item Although not illustrated by examples, their theorems allow more general sets $\Pi$.
    \item Their compensation condition uses a more general weight than the Sussmann weight of \eqref{eq:wtheta}.
\end{itemize}

Recall that $\B$ denotes a factor-parity Hall set on $X$ with $X_\good = \{ X_1, \dotsc, X_q \}$ and $X_\bad = \{ X_0 \}$.
Let $\Pi$ be an alphabet obtained from $(X,\B)$ by successive Hall-compatible Lazard eliminations, and let $E \subset \B$ be the set of eliminated Hall elements.
In other words, $E$ is an \emph{initial segment} of $\B$.
We assume (as in \cite{AgrachevGamkrelidze1993_Semigroups,Krastanov2009}) that
\begin{equation}
    \label{eq:Sbad-in-LiePi}
    S_\bad \subset \Lie(\Pi).
\end{equation}
Repeated application of \cref{lem:lazard} shows that $E$ is factor-stable, that $\Fac(\Pi) \subset E$, and that the Hall set induced on $\Pi$ is again of the factor-parity class (so $\Pi = \Pi_\good \sqcup \Pi_\bad$), the good/bad status of every non-eliminated Hall element being preserved.
We identify this final Hall set with $\B \setminus E$.
For $b \in \B \setminus E$, we denote by $\deg_\Pi(b)$ its total degree over $\Pi$ and by $n_\pi(b)$ the multiplicity of the letter $\pi \in \Pi$ in its expression over $\Pi$.
Since $\Lie(\Pi)$ is spanned by $\eval(\B \setminus E)$, an eliminated Hall element cannot lie in $\Lie(\Pi)$, so \eqref{eq:Sbad-in-LiePi} forces
\begin{equation}
    \label{eq:E-cap-Sbad}
    E \cap S_\bad = \varnothing.
\end{equation}

The definitions of the bad brackets to be compensated in \cite{AgrachevGamkrelidze1993_Semigroups,Krastanov2009} are as follows:
\begin{align}
    \label{eq:AGbad-app}
    \operatorname{AG}_\bad(\Pi)
    &:= \left\{ b \in (\B \setminus E) \cap S_\bad \;\middle|\; \deg_\Pi(b) \text{ is odd} \right\},
    \\
    \label{eq:Kbad-app}
    \operatorname{K}_\bad(\Pi)
    &:= \left\{ b \in \operatorname{AG}_\bad(\Pi) \;\middle|\; n_\pi(b) \text{ is even for every } \pi \in \Pi \setminus S_\bad \right\}.
\end{align}

\begin{lemma}
    \label{lem:E-good-app}
    One has $E \subset \B_\good$ and $\Pi_\bad \subset S_\bad$.
    In particular $\Pi \setminus S_\bad \subset \Pi_\good$.
\end{lemma}

\begin{proof}
    If $E \cap \B_\bad \neq \varnothing$, pick $e$ in it of minimal length.
    Since $E$ is factor-stable, $\Fac(e) \subset E$, so no factor of $e$ is bad, by minimality of $\abs{e}$.
    Hence $e \in S_\bad$ by \cref{cor:bad-in-Sbad}, contradicting \eqref{eq:E-cap-Sbad}.
    Thus $E \subset \B_\good$.
    Let now $\pi \in \Pi_\bad$: then $\Fac(\pi) \subset E \subset \B_\good$, so $\pi \in S_\bad$, again by \cref{cor:bad-in-Sbad}.
\end{proof}

\begin{proposition}
    \label{prop:three-inclusions}
    Under the assumptions above,
    \begin{equation}
        \label{eq:three-inclusions}
        \B_\bad^\irr
        \subset
        \operatorname{K}_\bad(\Pi)
        \subset
        \operatorname{AG}_\bad(\Pi)
        \subset
        S_\bad.
    \end{equation}
\end{proposition}

\begin{proof}
    The last two inclusions are immediate from \eqref{eq:AGbad-app}--\eqref{eq:Kbad-app}.
    Let $b \in \B_\bad^\irr$.
    By \cref{lem:bad-irreducible-decomposition}, $b \in S_\bad$, hence $b \notin E$ by \eqref{eq:E-cap-Sbad}, so $b$ is a bad Hall bracket over $\Pi$.
    Applying \cref{lem:bad-parity} \ref{it:parity-letters} to the factor-parity Hall set on $\Pi$ and using \cref{lem:E-good-app},
    \begin{equation}
        \label{eq:K-even-condition-app}
        n_\pi(b) \equiv 0 \pmod 2
        \qquad \text{for every } \pi \in \Pi_\good \supset \Pi \setminus S_\bad,
    \end{equation}
    which is the parity condition of \eqref{eq:Kbad-app}.
    Moreover, by additivity of the $X$-multidegree under substitution of the generators,
    \begin{equation}
        n_0(b)
        = \sum_{\pi \in \Pi} n_\pi(b) n_0(\pi)
        \equiv \sum_{\pi \in \Pi_\bad} n_\pi(b)
        \equiv \sum_{\pi \in \Pi} n_\pi(b)
        = \deg_\Pi(b)
        \pmod 2,
    \end{equation}
    the two congruences following from \eqref{eq:K-even-condition-app}, together with the oddness of $n_0(\pi)$ for $\pi \in \Pi_\bad \subset S_\bad$.
    Since $b \in S_\bad$, $n_0(b)$ is odd, hence so is $\deg_\Pi(b)$ and $b \in \operatorname{K}_\bad(\Pi)$.
\end{proof}

To summarize, Sussmann's symmetry argument singles out the parity set $S_\bad$.
Agrachev--Gamkrelidze first change the free generating family, then retain the elements of odd total degree in the new generators.
Krastanov retains in addition the parity of each individual generator outside the Sussmann set.
Factor-parity Hall sets make the same mechanism recursive: once a good generator has been eliminated, the induced alphabet is again split into good and bad generators, and the parity rule is applied anew.

\section{PBW triangularity and the dual Hall factorization}
\label{app:dual}

\subsection{PBW triangularity}
\begin{lemma}[PBW triangularity]
    \label{lem:app-pbw-triangular} Let $D=d_1\cdots d_k$ be a decreasing PBW-monomial and let $e_1,\dots,e_\ell\in\B$. 
    Assume that, for some $t\leq\min(k,\ell)$, $e_t > d_t$ and $e_j = d_j$ for all $j < t$.
    Then every PBW-monomial occurring in the PBW expansion of $e_1\cdots e_\ell$ is strictly larger than $D$.
\end{lemma}

\begin{proof}
    We argue lexicographically on $\ell$ and the number of inversions of the sequence $e_1,\dots,e_\ell$. 
    If the sequence is decreasing, then for every $c>e_t$ the multiplicities of $c$ in $D$ and in $e_1\cdots e_\ell$ agree, while the latter has strictly larger multiplicity at $e_t$. 
    Hence $D\prec e_1\cdots e_\ell$ by the definition of the PBW order. 
    Otherwise choose $i$ with $e_i<e_{i+1}$. 
    Straightening this adjacent inversion gives
    \begin{equation}
        e_1\cdots e_\ell = e_1\cdots e_{i+1}e_i\cdots e_\ell + \sum_{c\in\supp_\B[e_i,e_{i+1}]} \lambda_c\, e_1\cdots e_{i-1}c\,e_{i+2}\cdots e_\ell .
    \end{equation}
    Every occurring $c$ satisfies $c>e_i$ by \cref{lem:left_fact} and the third Hall axiom. 
    The first term has one fewer inversion, while every term in the sum has one fewer term. 
    It remains only to check that the hypothesis is preserved. 
    One cannot have $i\leq t-2$, since $e_i=d_i\geq d_{i+1}=e_{i+1}$ there. 
    If $i=t-1$, both kinds of terms have their first strict increase over $D$ already at position $t-1$. 
    If $i=t$, they have it at position $t$. If $i\geq t+1$, their first $t$ factors are unchanged. 
    The induction hypothesis therefore applies to every term.
\end{proof}

\subsection{Dual of a PBW-monomial}

\begin{proof}[Proof of \cref{lem:dual-product}]
    For $c\in\B$, $\Delta c=c\otimes\one+\one\otimes c$, and $c\otimes\one$ commutes with $\one\otimes c$.
    Hence, for a PBW-monomial~$R$,
    \begin{equation}
        \Delta R=\prod_{c\in\B}^{\searrow}\bigl(c\otimes\one+\one\otimes c\bigr)^{e_c(R)}
        =\sum_{A \odot B=R}\Bigl(\prod_{c\in\B}\binom{e_c(R)}{e_c(A)}\Bigr)A\otimes B,
    \end{equation}
    the sum being over pairs of PBW-monomials with $e(A)+e(B)=e(R)$. 
    Taking the scalar product with $S_P\otimes S_Q$ and using \eqref{eq:shuffle} gives the first identity.
    The second follows by iteration.
\end{proof}

\subsection{Dual of a Hall element}

We prove \cref{lem:dual-hall} for the convention of \cref{def:Hall}.

\begin{lemma}[Terminal contraction]
    \label{lem:app-contraction} Let $P=p_1\cdots p_m$ be a decreasing PBW-monomial with $m\geq1$, let $y\in X$, and define recursively $d_{m+1}:=y$ and, as long as $p_i<d_{i+1}$, $d_i:=(p_i,d_{i+1})$.
    Then:
    \begin{enumerate}[label=\textup{(\alph*)},leftmargin=2em]
        \item each $d_i$ so defined belongs to $\B$; \item if $p_i\geq d_{i+1}$ for some $i$, then no one-factor PBW-monomial occurs in $Py$; \item if $p_i<d_{i+1}$ for all $i\in\{1,\dots,m\}$, then the only one-factor PBW-monomial occurring in $Py$ is $d_1$, with coefficient $1$.
    \end{enumerate}
\end{lemma}
\begin{proof}
    \textup{(a)} By the second Hall axiom: $p_m<d_{m+1}=y\in X$, and, for $i<m$,
    \begin{equation}
        \lambda(d_{i+1})=p_{i+1}\leq p_i<d_{i+1}.
    \end{equation}
    \textup{(b)--(c)} We prove by descending induction on $i\in\{m,\dots,0\}$ that, as long as $d_{m+1},\dots,d_{i+1}$ are defined,
    \begin{equation}
        \label{eq:app-induction} Py=p_1\cdots p_i\,d_{i+1}+R_i,
    \end{equation}
    where every PBW-monomial occurring in $R_i$ has at least two factors. 
    For $i=m$, this holds with $R_m=0$. Assume \eqref{eq:app-induction} for some $i\geq1$.
    If $p_i\geq d_{i+1}$, then $p_1\geq\cdots\geq p_i\geq d_{i+1}$.
    In straightening $p_1\cdots p_i d_{i+1}$, the final factor $d_{i+1}$ is never involved: straightening two preceding factors either interchanges them or replaces them by a Hall element larger than the smaller one, hence still at least $d_{i+1}$. 
    Consequently every resulting PBW-monomial retains $d_{i+1}$ as its final factor and has at least two factors. 
    This proves \textup{(b)}. 
    Suppose instead that $p_i<d_{i+1}$. 
    Then $p_id_{i+1}=d_{i+1}p_i+d_i$ and therefore
    \begin{equation}
        p_1\cdots p_i\,d_{i+1} = p_1\cdots p_{i-1}d_{i+1}p_i + p_1\cdots p_{i-1}d_i.
    \end{equation}
    In the first summand every factor preceding the final $p_i$ is at least $p_i$. 
    By the same straightening argument, that final $p_i$ is never involved, so every PBW-monomial arising from the first summand has at least two factors. 
    This proves \eqref{eq:app-induction} at $i-1$. 
    At $i=0$ we obtain $Py=d_1+R_0$, where every PBW-monomial occurring in $R_0$ has at least two factors, which proves \textup{(c)}.
\end{proof}

\begin{lemma}[Recognition]
    \label{lem:app-recognition} Let $c\in\B\setminus X$ and write its factorization as $c = \ad_{c_r}^{m_r}\dotsb \ad_{c_1}^{m_1}(\seed(c))$.
    Set $M_c:=c_r^{m_r}\cdots c_1^{m_1}$.
    Then, for every PBW-monomial $P$ and every $y\in X$,
    \begin{equation}
        \pair{Py}{S_c} = \delta_{P,M_c}\,\delta_{y,\seed(c)}.
    \end{equation}
\end{lemma}

\begin{proof}
    For $P=\one$ both sides vanish, since $\abs{c}\geq2$. Let $P=p_1\cdots p_m$, $m\geq1$. By \cref{lem:app-contraction}, $\pair{Py}{S_c}\neq0$ forces all contractions to be defined and $c=d_1$, and then $\pair{Py}{S_c}=1$. Grouping equal consecutive factors of $P$ as $e_j^{r_j}\cdots e_1^{r_1}$, with $e_1<\cdots<e_j$, we get
    \begin{equation}
        c = \ad_{e_j}^{r_j}\cdots \ad_{e_1}^{r_1}(y), \qquad e_i\leq p_1<c,
    \end{equation}
    the last inequality by the third Hall axiom. This is the Hall factorization of $c$, so uniqueness in \cref{lem:factorization} gives $P=M_c$ and $y=\seed(c)$. Conversely, for $P=M_c$ and $y=\seed(c)$, the contractions of \cref{lem:app-contraction} rebuild $c=d_1$ with coefficient $1$.
\end{proof}

\begin{proof}[Proof of \cref{lem:dual-hall}]
    Let $c\in\B\setminus X$ and $M_c$ be as above. 
    For every word $u$ and every $y\in X$, expanding $u=\sum_P\pair{u}{S_P}P$, and using \cref{lem:app-recognition}, we obtain
    \begin{equation}
        \pair{uy}{S_c} = \sum_P \pair{u}{S_P}\pair{Py}{S_c} = \delta_{y,\seed(c)}\pair{u}{S_{M_c}} = \pair{uy}{S_{M_c}\seed(c)}.
    \end{equation}
    Every word of degree $\abs{c}\geq2$ is of the form $uy$, and $S_c$ and $S_{M_c}\seed(c)$ are $\deg$-homogeneous of that degree. 
    Hence $S_c=S_{M_c}\seed(c)$.
\end{proof}

\section*{Acknowledgments}

We acknowledge support from the Fondation Simone et Cino Del Duca -- Institut de France.

\bibliographystyle{plain}
\bibliography{control}

\end{document}